\documentclass[12pt]{iopart}
\usepackage[latin1]{inputenc}
\usepackage[english]{babel}
\usepackage{listings,relsize}
\usepackage{amsmath}
\usepackage{amsfonts}
\usepackage{amssymb}
\usepackage{eucal}
\usepackage{mathrsfs}
\usepackage{iopams}
\usepackage{graphicx}
\usepackage{epstopdf}
\usepackage[hang,small,bf]{caption}
\usepackage[labelformat=simple]{subcaption}

\usepackage{bm}
\usepackage{url}
\usepackage{booktabs}
\usepackage{hyperref}
\usepackage{harvard}
\usepackage{alltt}
\usepackage{color}
\usepackage{fullpage} 
\usepackage{geometry}
\definecolor{string}{rgb}{0.7,0.0,0.0}
\definecolor{comment}{rgb}{0.13,0.54,0.13}
\usepackage[disable]{todonotes}  		

\graphicspath{{figures/}}
\newcommand{\vare}{\varepsilon}
\newcommand{\be}{\begin{equation}\fl}
\newcommand{\ee}{\end{equation}}

\hypersetup{colorlinks=true,linkcolor=blue,pdfauthor= Elena Bossolini}

\newtheorem{theorem}{Theorem}[section]
\newtheorem{lemma}[theorem]{Lemma}
\newtheorem{proposition}[theorem]{Proposition}

\newtheorem{conjecture}[theorem]{Conjecture}
\newtheorem{remark}{Remark} 

\newenvironment{proof}[1][Proof]{\begin{trivlist}
\item[\hskip \labelsep {\bfseries #1}]}{\end{trivlist}}
\newenvironment{definition}[1][Definition]{\begin{trivlist}
\item[\hskip \labelsep {\bfseries #1}]}{\end{trivlist}}

\newcommand*{\qed}{\hfill\ensuremath{\Box}}%

\begin{document}
\title[Singular limit analysis of a model for earthquake faulting]{Singular limit analysis of a model for earthquake faulting}
\author{Elena Bossolini, Morten Br\o ns and Kristian Uldall Kristiansen}
\address{ Department of Applied Mathematics and Computer Science, Technical University of Denmark,
Kongens Lyngby 2800, DK}
\ead{\mailto{ebos@dtu.dk}, \mailto{mobr@dtu.dk}, \mailto{krkri@dtu.dk}}
\begin{abstract}
In this paper we consider a one dimensional spring-block model  describing earthquake faulting. By using geometric singular perturbation theory and the blow-up method  we provide a detailed description of the periodicity of the earthquake episodes. In particular we show that the limit cycles arise from a degenerate Hopf bifurcation whose degeneracy is due to an underlying Hamiltonian structure that leads to large amplitude oscillations. We use a Poincar\'e compactification  to study the system  near infinity. At infinity the critical manifold loses hyperbolicity with an exponential rate. We use an adaptation of the blow-up method to recover the hyperbolicity.  This enables the identification of a new attracting manifold that organises the dynamics at infinity. This in turn leads to the formulation of a  conjecture on the behaviour of the limit cycles as the time-scale separation increases. We illustrate our findings with numerics and suggest an outline of the  proof of this conjecture.
\end{abstract}
\pacs{ }

\submitto{\NL}
\maketitle

\section{Introduction}\label{sec:introduction}
Earthquake events are a non-linear multi-scale phenomenon. Some of the non-linear occurrences  are fracture healing, repeating behaviour and memory effects  \cite{Ruina1983,heaton1990a,vidale1994a,Marone1998}. In this paper we focus   on the repeating behaviour of the earthquake cycles, where a cycle is defined as the combination of a rupture event with a following healing phase. 
An earthquake rupture consists of the instantaneous slipping of a fault side relative to the other side.  The healing phase allows the fault to strengthen  again and this process evolves on a longer time scale than the rupture event \cite{carlson1989a,marone1998a}. 

 The repetition of the earthquake events is significant for the predictability of earthquake hazards. The data collected in the Parkfield experiment in California show evidence of recurring  micro-earthquakes \cite{Nadeau1999,Marone1995,bizzarri2010a,Zechar2012}. For large earthquakes it is harder to detect a repeating pattern from the data, even though recent works indicate the presence of recurring cycles \cite{Ben-zion2008}.  

The one dimensional spring-block model together with the empirical Ruina friction law  is a fundamental model to describe earthquake dynamics \cite{Burridge1967,Ruina1983,rice1983a,gu1984a,rice1986a,carlson1991a,belardinelli1996a,fan2014a}. Although  the model does not represent all the non-linear phenomena of an  earthquake rupture, it still reproduces the essential properties of the  fault dynamics as extrapolated from experiments on rocks. The dimensionless form of the model is:
\be\label{eq:3Dproblem}
\begin{aligned}
\dot{x} &= -\rme^{z}\left(x +(1+\alpha) z\right),\\
\dot{y} &= \rme^{z} -1,\\
\vare \dot{z} &= -\rme^{-z}\left(y + \frac{x+z}{\xi} \right).
\end{aligned}\ee  
Numerically, it has been observed that  \eqref{eq:3Dproblem} has periodic solutions corresponding to the recurrence of the earthquake episodes, as shown in Figure \ref{fig:3dnumerics}  for two different values of the parameter $\vare$ and $\alpha>\xi$  fixed. The steep growth of the $y$-coordinate corresponds to the earthquake rupture, while the slow decay corresponds to the healing phase.  Hence the periodic solutions of  \eqref{eq:3Dproblem} have a multiple time-scale dynamics. Furthermore in  Figure \ref{fig:3dnumerics} we observe that the amplitude of the oscillations  increases for decreasing values of the time-scale separation $\vare$. For these reasons extensive numerical simulations are difficult to perform  in the relevant parameter range, that is $\vare \in [10^{-24}, 10^{-8}]$  \cite{rice1986a,carlson1989a,madariaga1996a,lapusta2000a,Erickson2008,Erickson2011}.    
\begin{figure}[t!]
	\centering
	\begin{subfigure}{.47\textwidth}\caption{ }\label{fig:3dbehaviour}
  \centering
  \includegraphics[width=\linewidth]{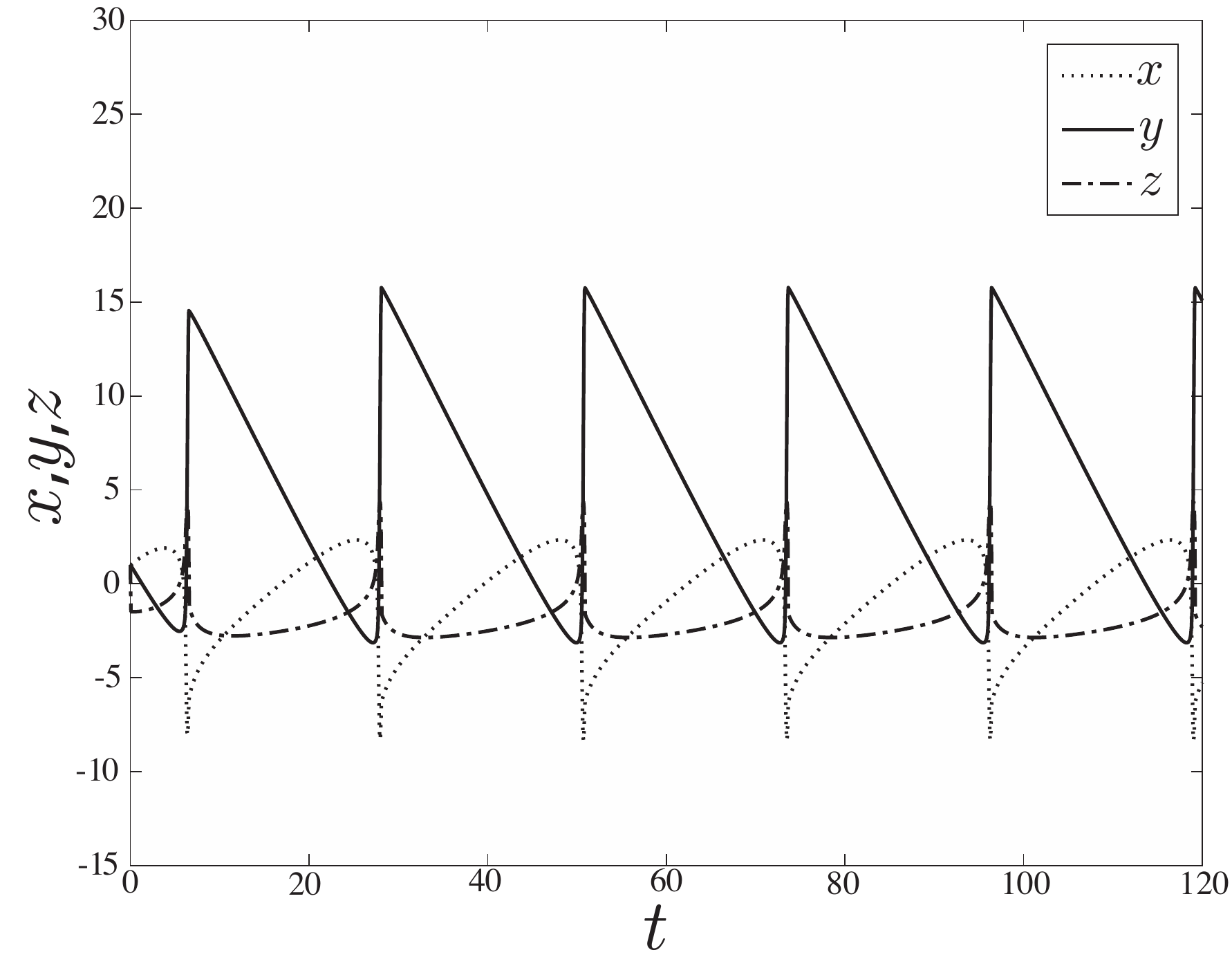}\end{subfigure}
\begin{subfigure}{.47\textwidth}\caption{ }\label{fig:3dbehaviour_eps}
  \centering
  \includegraphics[width=\linewidth]{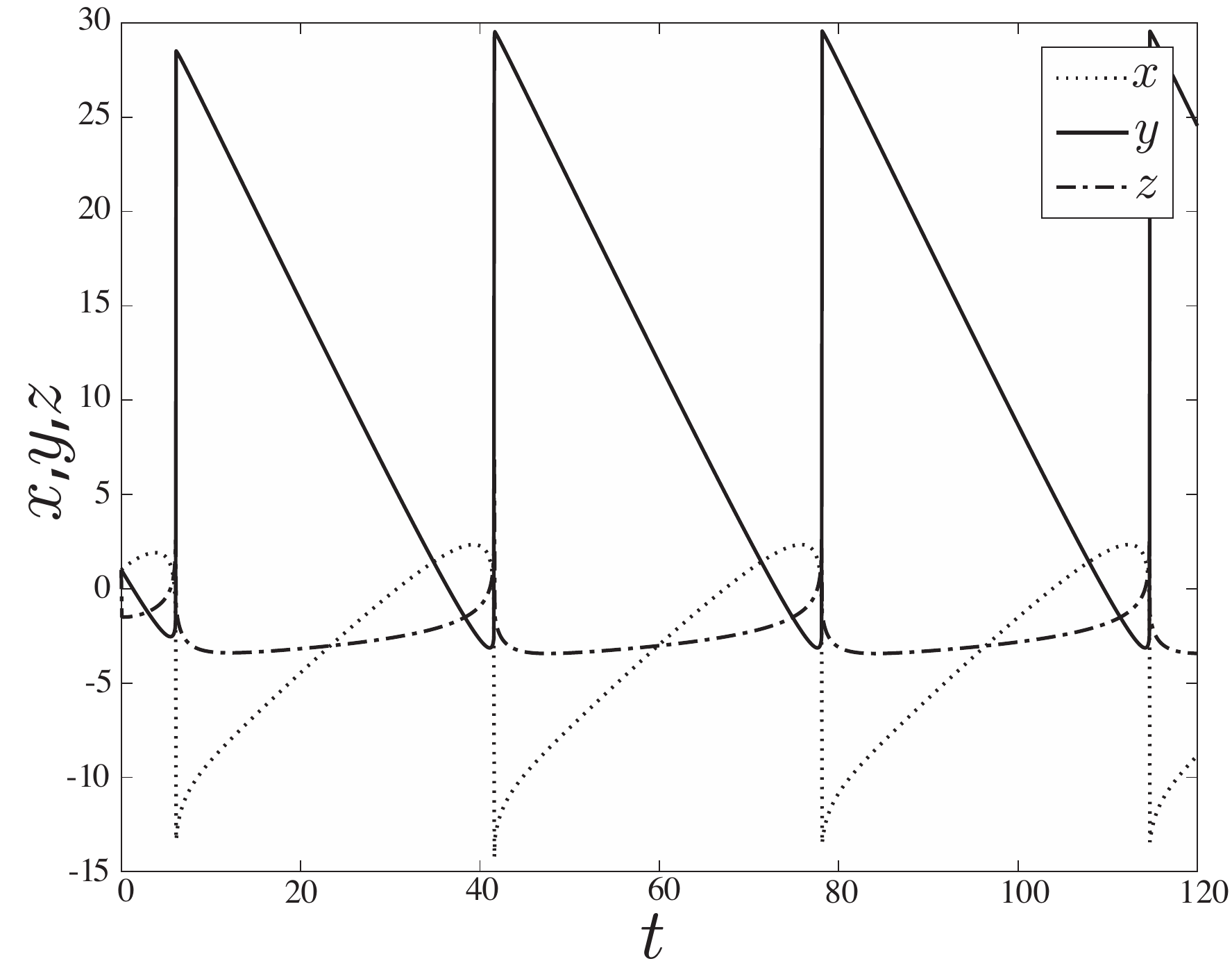}\end{subfigure}
  \begin{subfigure}{.63\textwidth}\caption{ }\label{fig:3d_plot_introduction}
  \centering
  \includegraphics[width=\linewidth]{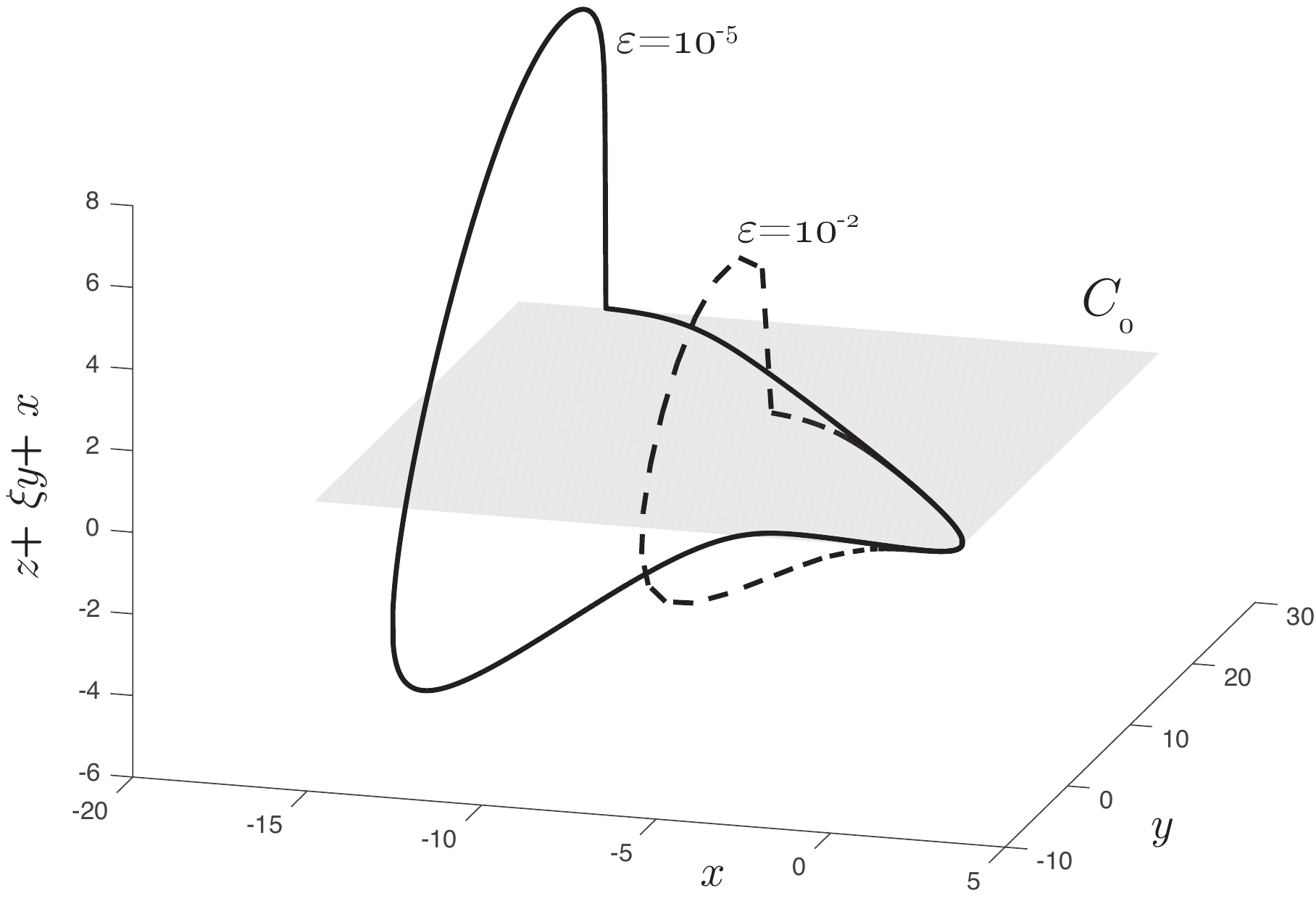}\end{subfigure}
  \caption{Numerical simulations of \eqref{eq:3Dproblem} for  $\alpha=0.9$ and $\xi=0.5$. $\vare=10^{-2}$ in \subref{fig:3dbehaviour}  while $\vare=10^{-5}$ in \subref{fig:3dbehaviour_eps}. In \subref{fig:3d_plot_introduction}    phase space of both simulations. The grey plane $C_0$ and the coordinate $z+\xi y + x$  are clarified  in section \ref{sec:GSPT}. }
	\label{fig:3dnumerics}
\end{figure}
We remark  that the periodic solutions of  \eqref{eq:3Dproblem} appear in a finite interval of values of $\alpha>\xi$. If $\alpha$ is much larger than $\xi$ then chaotic  dynamics emerges, as documented by \citename{Erickson2008} \citeyear{Erickson2008}.\\
 It is the purpose of the present paper to initiate a rigorous mathematical study of \eqref{eq:3Dproblem}  as a singular perturbation problem \cite{Jones1995,Kaper1999}.  At the singular limit $\vare=0$  we find an unbounded \textit{singular cycle} when $\alpha>\xi$. For $\vare>0$ we conjecture this cycle to perturb into a stable, finite amplitude limit cycle that explains the behaviour of Figure \ref{fig:3dnumerics}.  In this way we can predict the periodic solutions of \eqref{eq:3Dproblem} even in parameter regions that are not possible to explore numerically.  We expect that the deeper understanding of \eqref{eq:3Dproblem} that we provide, together with the techniques that we introduce, can be of  help to study the continuum formulation of the Burridge and Knopoff model, in particular regarding the analysis of the Heaton pulses \cite{heaton1990a}.

 As we will see in section \ref{sec:GSPT}, in our analysis the critical manifold loses normal hyperbolicity at infinity with an exponential rate. This is a non-standard loss of hyperbolicity  that also appears  in other  problems  \cite{rankin2011a}. To deal with this issue we will first introduce a compactification of the phase space with the Poincar\'e sphere \cite{chicone2006a} and  repeatedly  use the blow-up method of \citename{dumortier1996a} \citeyear{dumortier1996a} in the version of \citename{Krupa2001} \citeyear{Krupa2001}. In particular we will use a  technique that has been recently analyzed in \cite{Kristiansen2015a}. For an  introduction to the blow-up method we refer to \cite{Kuehn2015}.

Another way to study system  \eqref{eq:3Dproblem} when $\vare\ll1$ is by using the method of matched asymptotic expansions, see  \cite{eckhaus1973a} for an introduction.  \citename{putelat2008a} \citeyear{putelat2008a} have done the matching of the different time scales of \eqref{eq:3Dproblem} with an energy conservation argument, while in  \cite{pomeau2011critical} the causes of the  switch between the two different time scales are not studied.  However, the relaxation oscillation behavior  of the periodic solutions of  \eqref{eq:3Dproblem} is not explained.

Our paper is structured as follows. In section \ref{sec:model} we briefly discuss the physics of system \eqref{eq:3Dproblem}. In section \ref{sec:GSPT} we set \eqref{eq:3Dproblem} in the formalism of geometric singular perturbation theory and in section \ref{sec:singular_analysis} we consider the analysis of the reduced problem for $\alpha=\xi$ and $\vare=0$. Here a degenerate Hopf bifurcation appears whose degeneracy is due to an underlying Hamiltonian structure that we identify. We derive a bifurcation diagram in section \ref{sec:compact_reduced}   after having introduced a  compactification of the reduced problem. From this and from the analysis of section \ref{sec:perturbed_reduced}, we  conclude that the limit cycles of Figure \ref{fig:3dnumerics} cannot be described by the sole analysis of the reduced problem. In   section \ref{sec:infinity} we  define a candidate  \textit{singular cycle} $\Gamma_0$ that is used in our main result,  Conjecture \ref{con:conjecture_thm}. This conjecture is on the existence of limit cycles $\Gamma_\vare\to\Gamma_0$ for $\vare\ll1$. The conjecture is  supported by numerical simulations but in sections  \ref{sec:proof} and \ref{sec:global_results} we also lay out the foundation of a proof by using the blow-up method to gain  hyperbolicity of  $\Gamma_0$.     Finally in section \ref{sec:conclusion} we conclude and summarize the results of our analysis.



\section{Model}\label{sec:model}
The one dimensional spring-block model is presented in Figure \ref{fig:Burridge_Knopoff_model}. We suppose that one fault side slides at a constant velocity $v_0$ and drags the other fault side of mass $M$ through a spring of stiffness $\kappa$. The friction force $F_{\mu} = \sigma\mu$ acts against the motion. A common assumption is to suppose that the normal stress $\sigma$, i.e. the stress normal to the friction interface  \cite{Nakatani2001}, is constant $\sigma=1$. The friction coefficient $\mu$ is modelled with the Ruina rate and state friction law  $\mu=\mu(v,\theta)$, with $v$ the sliding velocity and $\theta$ the {\it state} variable. The state $\theta$ accounts for how long the  two surfaces have been in contact \cite{Ruina1983,Marone1998}.
\begin{figure}[t!]
\centering
\includegraphics[scale=0.9]{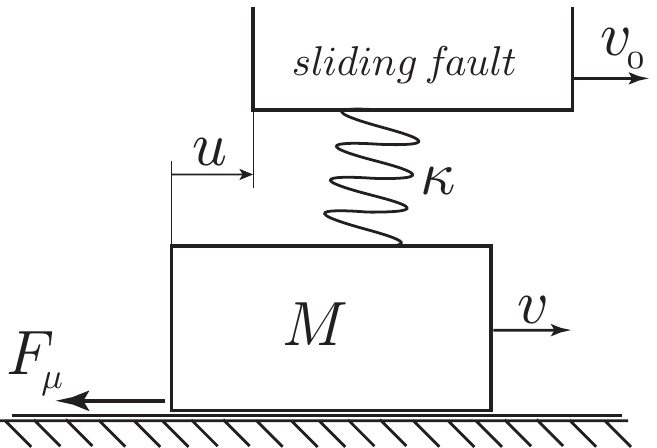}
\caption{Spring-block model describing the earthquake faulting.}\label{fig:Burridge_Knopoff_model}
\end{figure}
The equations of our model are:
\be\begin{aligned}\label{eq:ericksonequations}
{\theta}' &= -\frac{v}{L}\left(\theta + b \ln\left(\frac{v}{v_0}\right) \right),\\
{u}' &= v-v_0,\\
M{v}' &= -  \kappa u -\left(\theta + a \ln \left(\frac{v}{v_0}\right) \right),
\end{aligned}\ee
where the variable $u$ is the relative displacement between the two fault sides and the prime denotes the time derivative. The parameter $L$ is the characteristic displacement that is needed to recover the contact between the two surfaces when the slip occurs, while $a$ and $ b$ are empirical coefficients that depend on the material properties \cite{Marone1998}. 
We introduce the   dimensionless coordinates $(x, y, w, t)$ into system \eqref{eq:ericksonequations}, where  $\theta = ax ,\,\, u = L y,\,\, v=v_0 w,\,\, t= ({v_0}/L)t'$: 
\be\label{eq:3Dlog}
\begin{aligned}
\dot{x} &=-w\left(x +(1+\alpha) \ln(w)\right),\\
\dot{y} &= w -1,\\
\vare \dot{w} &= - y - \frac{x+\ln(w)}{\xi} .
\end{aligned}\ee
We notice that equation  \eqref{eq:3Dlog} has a singularity in $w=0$ and to avoid it we henceforth introduce the variable $z = \ln(w)$ so that we obtain the formulation presented in \eqref{eq:3Dproblem}.
In system \eqref{eq:3Dlog} we have introduced the parameters: $\vare = Mv_0^2/(\kappa L^2)$ such that $1/\sqrt{\vare}$ is a non-dimensional frequency, $\xi = (\kappa L)/a$: the non-dimensional spring constant and $\alpha= (b-a)/b$ describing the sensitivity to the velocity relaxation \cite{Erickson2008}. We consider the parameter values presented by  \citename{Madariaga1998} \citeyear{Madariaga1998}: $\vare \in [10^{-24}, 10^{-8}]$, $\xi=0.5$, $\alpha>\xi$. An extensive reference to the parameter sets is in  the work of \citename{dieterich1972} \citeyear{dieterich1972,dieterich1978a,Dieterich1979}. 
We choose to keep the parameter $\xi>0$ fixed (selecting $\xi=0.5$ in our computations) and we use $\alpha$ as the bifurcation parameter.  With  this choice the  study of  \eqref{eq:3Dproblem} as a singular perturbation problem is simplified. Indeed as we will see in section \ref{sec:GSPT}, the critical manifold of \eqref{eq:3Dproblem} is a surface that depends on  $\xi$. The results of our analysis can be easily interpreted for the case of $\alpha$ fixed and $\xi$ varying, that is the standard approach in the literature.

\section{Singular perturbation approach to the model}\label{sec:GSPT}
The positive constant $\vare\ll1$ in system \eqref{eq:3Dproblem} measures the separation of two  time scales.  In particular the variables $(x,y)$ are slow  while $z$ is fast.
We call equation \eqref{eq:3Dproblem} the slow problem and the dot refers to the  differentiation with respect to the slow time $t$. We introduce the fast time  $\tau = t/\vare$ to obtain the fast problem:
\be\label{eq:fast_problem}
\begin{aligned}
x' &= -\vare \rme^{z}(x +(1+\alpha) z),\\
y' &= \vare \left(\rme^{z} -1\right),\\
z' &= -\rme^{-z}\left(y + \frac{x+z}{\xi} \right),
\end{aligned}\ee 
where the  prime stands for differentiation with respect to  $\tau$.
The two systems \eqref{eq:3Dproblem} and  \eqref{eq:fast_problem} are equivalent whenever $\vare>0$. In the singular analysis  we consider  two different limit systems.  By setting $\vare=0$ in \eqref{eq:3Dproblem} we obtain the reduced problem:
\be \label{eq:reduced_problem_stupid}
\begin{aligned}
\dot{x} &= -\rme^{z}(x +(1+\alpha) z),\\
\dot{y} &=  \rme^{z} -1,\\
0 &= -\rme^{-z}\left(y + \frac{x+z}{\xi} \right),
\end{aligned}\ee
that is also referred in the  literature as the quasi-static slip motion (specifically $M\to0$ in \eqref{eq:ericksonequations}, \cite{Ruina1983}). Setting $\vare=0$ in  \eqref{eq:fast_problem} gives the layer problem:
\be \label{eq:layer_problem}
{z}' = -\rme^{-z}\left(y + \frac{x+z}{\xi} \right),
\quad (x,y)(\tau)=(x^0,y^0).\ee
System \eqref{eq:layer_problem} has a plane of fixed points that we denote the critical manifold:  
\be \label{eq:critical_manifold}
C_0  :=\Bigl\{(x, y, z) \in \mathbb{R}^3\Bigr| \quad  z= -x -\xi y \Bigr\}.
\ee
This manifold, depicted in grey in Figure \ref{fig:3d_plot_introduction}, is attracting:
\be \label{eq:normally_attracting} \frac{\partial {z}'}{\partial z}\biggr|_{C_0} = -\xi^{-1}\rme^{-z} <0. \ee
The results by \citename{fenichel1974a} \citeyear{fenichel1974a,Fenichel1979} guarantee that close to $C_0$ there is an attracting (due to \eqref{eq:normally_attracting}) slow-manifold $S_\vare$ for any compact set $(x,y)\in\mathbb{R}^2$ and $\vare$ sufficiently small.  However we notice in \eqref{eq:normally_attracting} that $C_0$ loses its normal hyperbolicity at an exponential rate when $z\to+\infty$. This is a key complication: orbits leave a neighborhood of the critical manifold even if it is formally attracting. This is a non-standard loss of hyperbolicity that appears also in other physical problems  \cite{rankin2011a}. To our knowledge, \cite{Kristiansen2015a} is the first attempt on a theory of exponential loss of hyperbolicity. In section \ref{sec:proof} we will apply the method described in \cite{Kristiansen2015a} to resolve the loss of hyperbolicity at infinity. In this paper we do not aim to give a general geometric framework to this approach. In the case of loss of hyperbolicity at an algebraic rate, like in the autocatalator problem studied originally by \citename{Gucwa2009783} \citeyear{Gucwa2009783}, we refer to the work of  \citename{kuehn2014a} \citeyear{kuehn2014a}.
  \\
Na\"ively we notice that when $z\gg1$ the dynamics of system \eqref{eq:3Dproblem} is driven by a new time scale, that is not related to its slow-fast structure. Assuming $z \gg \ln\vare^{-1}$ we can rewrite \eqref{eq:3Dproblem} as:
\be\label{eq:dynamicsinfinity}
 \begin{aligned}
\dot{x} &= - x -(1+\alpha)z,\\
\dot{y} &=1,\\
\dot{z} &= 0,
\end{aligned}\ee 
where we have further rescaled the time by dividing the right hand side by $\rme^z$ and ignored the higher order terms.
Hence in this regime there is a  family of $x$-nullclines:
\be \label{eq:L0}  x+(1+\alpha)z=0,\ee
that are attracting since:
\[ \frac{\partial \dot{x}}{\partial x} = -1.\]
This na\"ive approach is similar to the one used by \citename{rice1986a} \citeyear{rice1986a} to describe the different  time scales that appear in system \eqref{eq:3Dproblem}.

\section{Reduced Problem}\label{sec:singular_analysis}
We write the  reduced problem \eqref{eq:reduced_problem_stupid} as a vector field $f_0(y,z; \alpha)$  by eliminating $x$ in \eqref{eq:reduced_problem_stupid}:
\be \label{eq:reduced_problem}
f_0(y,z; \alpha) :=\quad \begin{cases}
\dot{y} &= \rme^{z} -1,\\
\dot{z} &= \xi + \rme^{z}\left(\alpha z - \xi y - \xi \right).
\end{cases}\ee
The following proposition  describes the degenerate Hopf bifurcation at the origin of \eqref{eq:reduced_problem} for $\alpha=\xi$.

\begin{proposition}\label{prop:hamiltonian}
The vector field \eqref{eq:reduced_problem} has a unique fixed point in $(y,z)=(0,0)$ that undergoes a degenerate Hopf bifurcation for $\alpha=\xi$. In particular $f_0(y,z; \xi)$ is Hamiltonian and it can be rewritten as:
\be \label{eq:hamiltonian_sys}
f_0(y,z; \xi) = g(y,z) J \nabla H(y,z),\ee
with 
\begin{subequations}\label{eq:g_H}
 \begin{flalign}
g(y,z) &=  \frac{\rme^{\xi y + z}}{\xi}, \label{eq:g_hamilt} \\
H(y,z) &= -\rme^{-\xi y}\left(\xi y - \xi z + \xi +1 - \xi\rme^{-z} \right) +1,\label{eq:H_hamilt} 
\end{flalign}
\end{subequations}
 and where $ J $ is the standard symplectic structure matrix: $J= \begin{bmatrix} 0 & 1 \\ -1 & 0 \end{bmatrix}. $
\end{proposition}

\begin{proof}
The linear stability analysis of  \eqref{eq:reduced_problem} in the fixed point $(y,z)=(0,0)$ gives the following Jacobian matrix:
\be\label{eq:jacobian_hopf}Df_0(0,0;\alpha) = \begin{bmatrix}   0 & 1\\ -\xi &\alpha-\xi 
\end{bmatrix}.  \ee
The trace of \eqref{eq:jacobian_hopf} is zero for $\alpha=\xi$ and its determinant is $\xi>0$. Hence a  Hopf bifurcation occurs.    
The direct substitution of \eqref{eq:g_H} into \eqref{eq:hamiltonian_sys} shows that system  \eqref{eq:reduced_problem} is Hamiltonian  for $\alpha=\xi$. Therefore the Hopf bifurcation is degenerate. 
\qed\end{proof} 
 \begin{figure}[!t]
	\centering
	\begin{subfigure}{.48\textwidth}
  \centering
 \caption{  } \label{fig:hamiltonian}
  \includegraphics[width = 0.66 \linewidth]{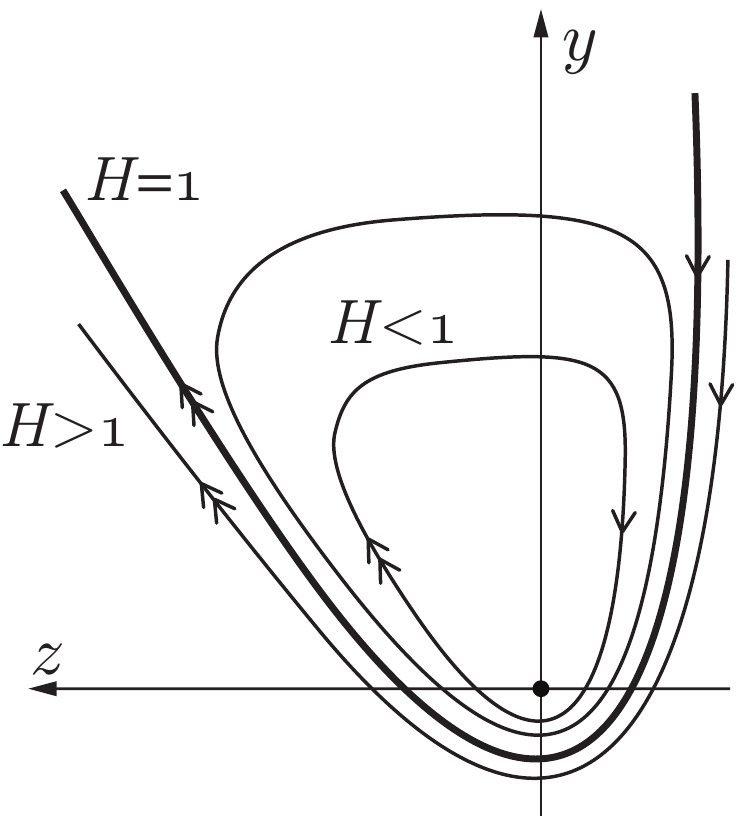}\end{subfigure}
\begin{subfigure}{.48\textwidth}
  \caption{ }\label{fig:Hamiltonian_solution_med}
  \centering \includegraphics[width=0.95\linewidth]{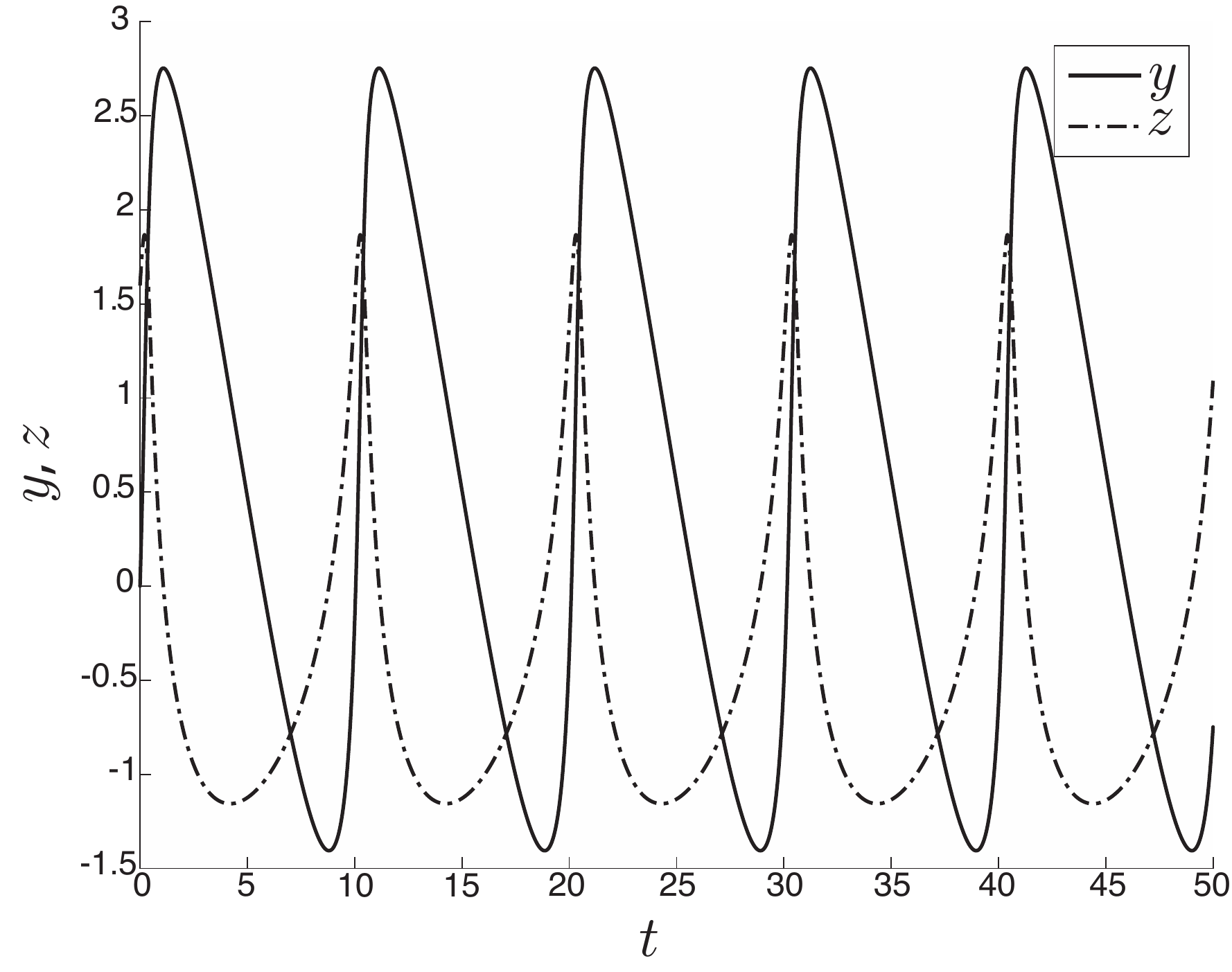} 
\end{subfigure}\caption{Behaviour of the reduced problem \eqref{eq:hamiltonian_sys} for $\alpha=\xi$. In  \subref{fig:hamiltonian}: phase space. The axis orientation is chosen in order to be consistent with the remaining figures of the paper. In \subref{fig:Hamiltonian_solution_med}: simulation of \eqref{eq:hamiltonian_sys} for $H=0.4,\,\,\xi=0.5$. }
	\label{fig:Hamiltonian_numerics}
	\end{figure}
The Hopf bifurcation of \eqref{eq:reduced_problem}  for $\alpha=\xi$  is a known result \cite{Ruina1983,putelat2008a,Erickson2008}. The  function $H(y,z)$ has been used as a Lyapunov function in \cite{gu1984a}  without realising the Hamiltonian structure of \eqref{eq:reduced_problem}.\\
 \noindent From Proposition \ref{prop:hamiltonian} we obtain a vertical family of periodic orbits for $\alpha=\xi$. 
The phase space of  \eqref{eq:hamiltonian_sys} is illustrated in Figure \ref{fig:hamiltonian} for positive values of $H(y,z)$. We remark that the fixed point $(y,z)=(0,0)$ is associated with $H(y,z)=0$. \\
The intersection of the $y$-axis with the orbits  $H(y,z)=h$ corresponds to the real roots of the Lambert equation: 
\be\label{eq:lambert} -\rme^{-\xi y}(\xi y +1) +1 = h, \quad h\geq0. \ee 
Equation \eqref{eq:lambert} has a real root for any $h>0$ in the region $y<0$, while a second real root in the region $y>0$ exists only  for $h\in(0,1)$ \cite{corless2014a}.
The intersection of the Hamiltonian trajectories with the $y$-axis is transversal for all $h>0$, since the following  condition holds:
\be\label{eq:transversality}
\frac{\partial H}{\partial y}(y,0) = \xi^2 y \rme^{-\xi y} \neq 0, \quad \forall y\neq0.
\ee
The trajectory  identified with $H(y,z)=1$ (that is   in bold in Figure \ref{fig:hamiltonian}) plays a special role since it separates the closed orbits for $H\in(0,1)$ from  the unbounded ones for $H\geq1$. Our analysis supports the results of \citename{gu1984a} \citeyear{gu1984a} and contrasts \cite{ranjith1999a}    where it is  claimed that   \eqref{eq:hamiltonian_sys} has no unbounded solutions.  
\begin{remark} From \eqref{eq:transversality} it follows that the function $H(y,0)$  defines a diffeomorphism between the points on the positive $y$-axis and the corresponding values $h\in(0,1)$.
\end{remark}
Figure \ref{fig:Hamiltonian_solution_med} highlights that the reduced problem \eqref{eq:hamiltonian_sys} has an intrinsic slow-fastness.
Indeed the phase space of \eqref{eq:hamiltonian_sys} is swept with different speeds depending on the region considered. This feature is represented in Figure \ref{fig:hamiltonian}, with the double arrow representing fast motion.  In particular when $ z>0$ the trajectories are swept faster than for $z<0$. This is due to the exponential function in \eqref{eq:reduced_problem}. The fast sweep for $z>0$ corresponds  to the steep increase in the $y$ coordinate of Figure \ref{fig:Hamiltonian_solution_med}. This fast dynamics for $z>0$ resembles the slip that happens during an earthquake rupture, while the slow motion for $z<0$ matches  the healing phase, recall Figure \ref{fig:3dnumerics}. From this observation we tend to disagree  with the notation used in the literature, that calls the reduced problem the quasi-static slip phase \cite{Ruina1983}.\\  
In order to describe the unbounded trajectories with $H(y,z)\geq 1$ for $y,z\to\infty$ and to extend the analysis to the case $\alpha\ne\xi$,  we introduce a compactification  of the reduced problem \eqref{eq:reduced_problem} and  then we  rewrite \eqref{eq:reduced_problem} on the Poincar\'e sphere.
 
\section{Compactification of the reduced problem}\label{sec:compact_reduced}

We define the  Poincar\'e sphere $\mathcal{S}^{2,+}$  as:
\be \label{eq:sphere_def} 
\mathcal{S}^{2,+} := \{(Y,Z,W)\in \mathbb{R}^3 \bigr|\quad  Y^2+Z^2+W^2=1, \quad W\geq0\},\ee
which projects  the phase space of \eqref{eq:reduced_problem}   onto  the northern hemisphere of $\mathcal{S}^{2,+}$.  We refer to \cite{chicone2006a} for further details on the compactification of vector fields. Geometrically \eqref{eq:sphere_def} corresponds to embedding \eqref{eq:reduced_problem} into the plane $W=1$ that we call the directional chart $k_2$:
\[ \fl k_2:= \mathcal{S}^{2,+} \cap \{W=1\}, \quad  y_2 = \frac{Y}{W}, \,\,z_2 = \frac{Z}{W},\]
and the dynamics on chart $k_2$ follows directly from \eqref{eq:reduced_problem} by variable substitution:
\be \label{eq:reduced_problem_2}
\begin{aligned}
\dot{y}_2 &= \rme^{z_2} -1,\\
\dot{z}_2 &= \xi + \rme^{z_2}\left(\alpha z_2 - \xi y_2 - \xi \right).
\end{aligned}\ee
The points at infinity in $k_2$  correspond to the condition $W=0$, that is the equator of $\mathcal{S}^{2,+}$.
To study the dynamics on the equator we introduce the two  additional directional charts:
\begin{subequations}
\label{eq:directional_charts}
 \begin{flalign}
&k_3:= \mathcal{S}^{2,+} \cap \{Z=1\}, \quad   y_3 = \frac{Y}{Z}, w_3 = \frac{W}{Z}, \label{eq:chart_kappa3def} \\
&k_1:= \mathcal{S}^{2,+} \cap \{Y=1\}, \quad  z_1 = \frac{Z}{Y}, w_1 = \frac{W}{Y}.\label{eq:chart_kappa1def} 
\end{flalign}
\end{subequations}
We follow the standard convention of \citename{Krupa2001} \citeyear{Krupa2001} and use the subscript $i = 1,2,3$ to denote a quantity  in chart $k_i$. 
We denote with $k_{ij}$  the transformation from chart $k_i$ to chart $k_j$ for  $i,j = 1,2,3$. We have the following change of coordinates: 
\begin{subequations}
\label{eq:change_charts}
 \begin{flalign}
&k_{23}:\quad  w_3= z_2^{-1}, \quad y_3 =y_2 z_2^{-1},\label{eq:R&S_kappa23} \\
&k_{21}:\quad  w_1 = y_2^{-1}, \quad z_1 = z_2y_2^{-1},\label{eq:R&S_kappa21} \\
&k_{31}:\quad w_1 = w_3 y_3^{-1}, \quad z_1= y_3^{-1},\label{eq:R&S_kappa31} 
\end{flalign}
\end{subequations}
that  are defined for $z_2>0$, $y_2>0$ and $y_3>0$ respectively. The inverse transformations $k_{ji}=k_{ij}^{-1}$ are defined similarly.  Figure \ref{fig:disegno_sfera} shows a graphical representation of the sphere and the directional charts.  \\
We define $C_{0,\infty}$ as the extension of the critical manifold $C_0$ onto the equator of the sphere. From \eqref{eq:normally_attracting} it follows that   $C_{0,\infty}$ is non-hyperbolic.   
\begin{figure}[t!]
\centering
\includegraphics[scale=0.38]{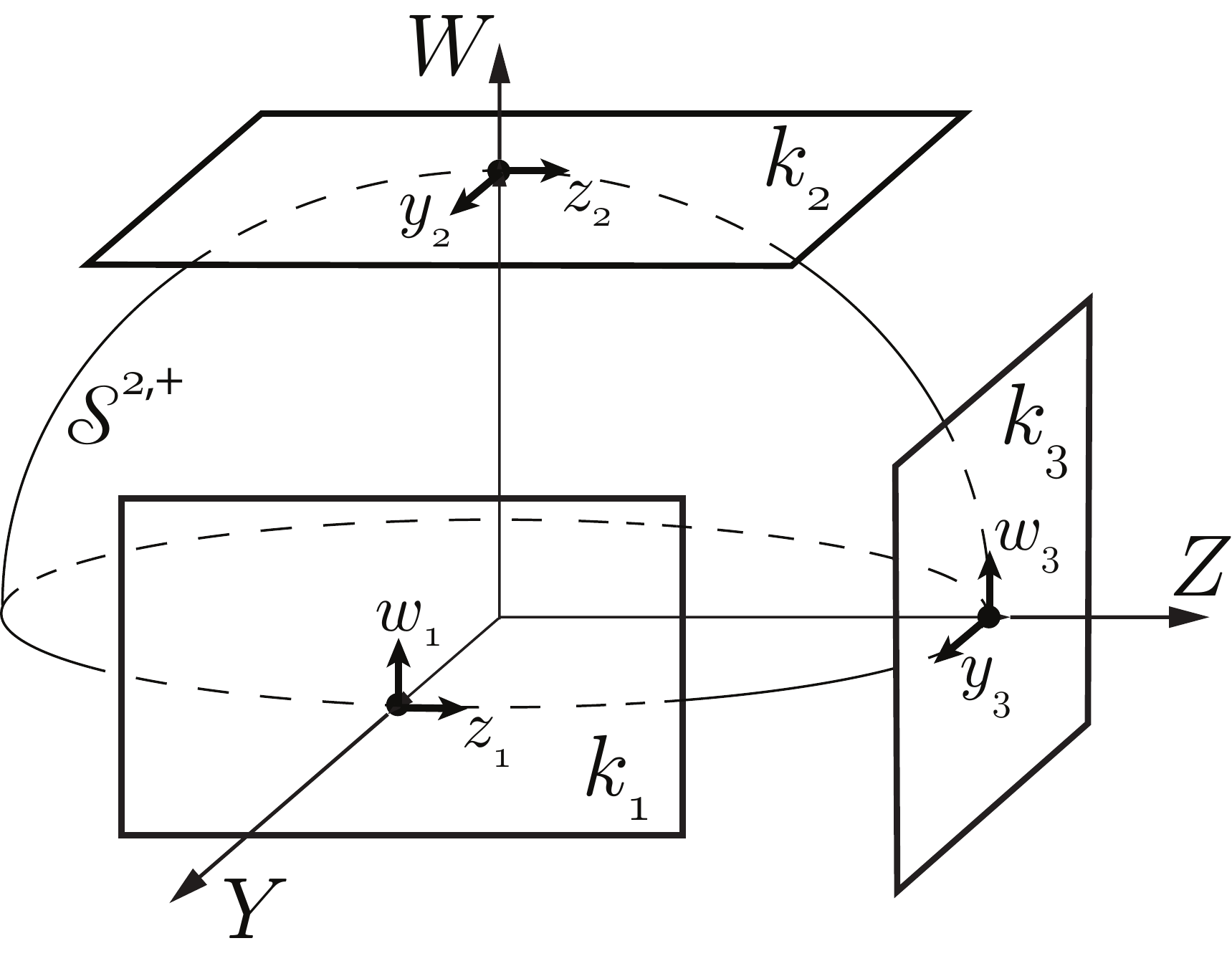}
\caption{\mbox{Poincar\'e sphere $\mathcal{S}^{2,+}$ and the directional charts $k_{1,2,3}$.}}\label{fig:disegno_sfera}
\end{figure} 

\begin{proposition}\label{prop:fixed_point_infinity}
There exists a time transformation   that is smooth  for $W>0$ and   that de-singularizes the dynamics  within $W=0$, so that the reduced problem \eqref{eq:reduced_problem} has  four fixed points $Q^{1,3,6,7}$ on $C_{0,\infty}$ satisfying: 
\begin{itemize}
\item $Q^1$ is an improper stable node with a single eigenvector tangent to $C_{0,\infty}$.
\item $Q^3$ has one unstable direction that is tangent to $C_{0,\infty}$ and a unique center-stable manifold $W^{c,s}$.
\item $Q^6$ has one  stable direction that is  tangent to $C_{0,\infty}$  and a unique center-unstable manifold $W^{c,u}$.
\item $Q^7$ is an improper unstable node with a single eigenvector tangent to $C_{0,\infty}$. 
\end{itemize}
The stability properties of the fixed points are independent of   $\alpha$, in particular both $W^{c,s}$ and $W^{c,u}$ are smooth in $\alpha$. 
\end{proposition}
\begin{figure}[ht!]
\centering
\includegraphics[scale=0.9]{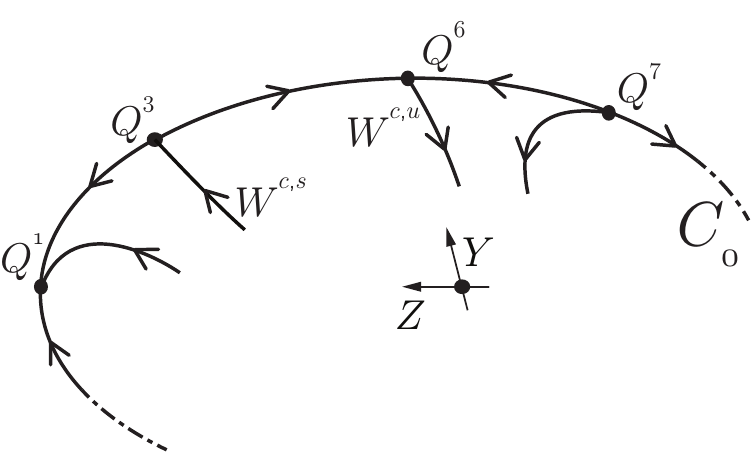}
\caption{Fixed points on the compactified critical manifold $C_0$. }\label{fig:fixed_point_infinity}
\end{figure} 
Figure \ref{fig:fixed_point_infinity} gives a representation of the statements of Proposition \ref{prop:fixed_point_infinity}. 
We remark that we use superscripts as enumeration of the points $Q^m, m=1,3,6,7$ to avoid confusion with the subscripts that we have used to define the charts $k_i, i=1,2,3$. In particular the enumeration choice of the superscripts will become clear in section \ref{sec:infinity}, where we will introduce the remaining  points $Q^{2,4,5}$ in \eqref{eq:QaQbQc}. In  Proposition \ref{theo:bifurcation_alpha} we relate the structure  at infinity of  \eqref{eq:reduced_problem} to the dynamics on  $C_0$ with respect to the  parameter $\alpha$. 
\begin{proposition}\label{theo:bifurcation_alpha}
Fix $c>0$ small and consider the parameter interval:
\be \label{eq:alpha_interval} \alpha\in[ \xi-c,\xi+c].\ee 
Then Figure \ref{fig:bifurcation_alpha} describes the phase space of \eqref{eq:reduced_problem} with respect to  $\alpha$. In particular:
\begin{itemize}
\item When $\alpha<\xi$ the set $W^{c,s} $ separates the basin of attraction of $(y,z)=(0,0)$ from the solutions that are forward asymptotic to $Q^1$.
\item When $\alpha=\xi$ Proposition \ref{prop:hamiltonian} holds. The set $H=1$  corresponds to  $W^{c,s}\cap W^{c,u}$. 
\item  When $\alpha>\xi$ the set $W^{c,u} $ separates the solutions that are backwards asymptotic to the origin to the ones that are backwards asymptotic to $Q^7$.
\end{itemize}
Therefore no limit cycles appear in the reduced problem for $\vare=0$ and $\alpha\neq\xi$.
\end{proposition}
 \begin{figure}\centering
 \begin{subfigure}{.32\textwidth}
  \centering 
  \caption{ }\label{fig:alpha_less_xi}
\includegraphics[width=\linewidth]{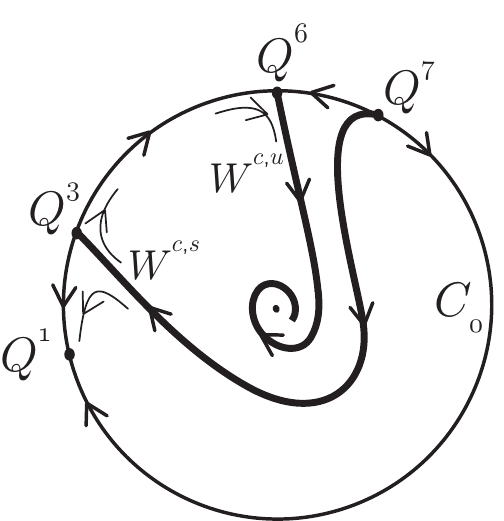}
\end{subfigure}
 \begin{subfigure}{.32\textwidth}
  \centering 
  \caption{ }\label{fig:alpha_equal_xi}
\includegraphics[width=\linewidth]{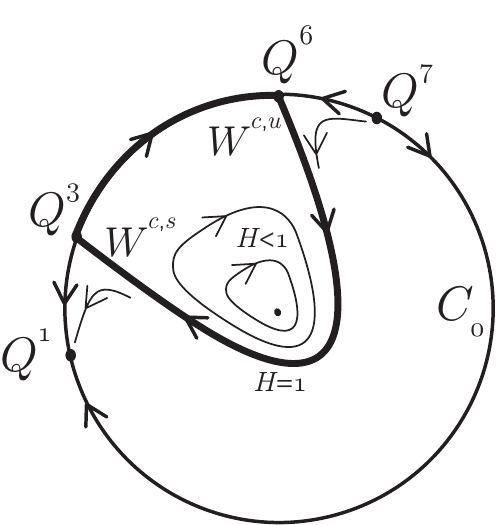}
\end{subfigure}
\begin{subfigure}{.32\textwidth}
  \centering 
  \caption{ }\label{fig:alpha_larger_xi}
\includegraphics[width=\linewidth]{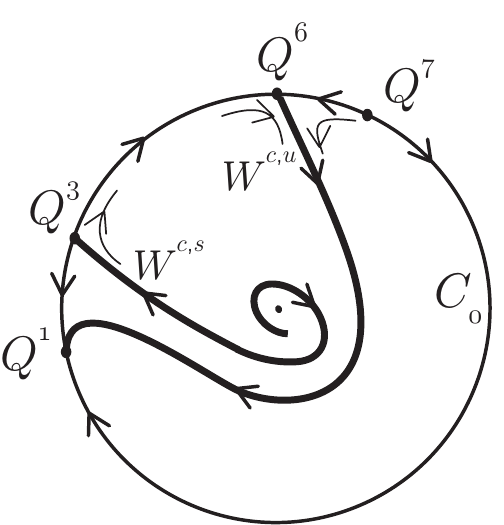}
\end{subfigure}
\caption{Bifurcation diagram of \eqref{eq:reduced_problem}  with respect to the parameter $\alpha$. Orbits spiral inwards for $\alpha<\xi$ \subref{fig:alpha_less_xi}  or outwards for $\alpha>\xi$ \subref{fig:alpha_larger_xi}. In \subref{fig:alpha_equal_xi}: $\alpha=\xi$.}\label{fig:bifurcation_alpha}
\end{figure}
\begin{remark}
The local stability analysis of $(y,z)=(0,0)$ can be directly obtained using $H(y,z)$ as a Lyapunov function. This was done in \cite{gu1984a}.  
\end{remark}
In the rest of the section we prove the previous two propositions. In  sections \ref{subsec:K3} and \ref{subsec:K1} we perform an analysis of \eqref{eq:reduced_problem} in the two charts $k_3$ and $k_1$ respectively to   show Proposition \ref{prop:fixed_point_infinity}. We prove Proposition \ref{theo:bifurcation_alpha} in section \ref{sec:infinity_to_reduced}.
 
\subsection{Chart $k_3$}\label{subsec:K3}
We insert   \eqref{eq:chart_kappa3def} into the reduced problem \eqref{eq:reduced_problem_2}  and obtain the following system:
 \be\label{eq:compactified_reduced_K3} \begin{aligned}
\dot{w}_3 &= -w_3(\alpha-\xi y_3)+\xi w_3^2 (1-\rme^{-\frac{1}{w_3}}),\\
\dot{y}_3 &= - y_3(\alpha-\xi y_3) -w_3(1+\xi y_3) (1-\rme^{-\frac{1}{w_3}}),
\end{aligned} \ee 
here we have divided the right hand side by $\exp(1/w_3)$ to de-singularize $w_3=0$.

\begin{remark}
The division by $ \exp(1/w_3)$ in \eqref{eq:compactified_reduced_K3}  is formally performed by introducing  the new time $t_3$ such that: 
\be\label{eq:timet3} \rmd t_3  = \exp(1/w_3)  \rmd t.\ee  
A similar de-singularization procedure is also used in the blow-up method. 
\end{remark}
System \eqref{eq:compactified_reduced_K3} has two fixed points:
\begin{subequations}
\label{eq:Q1_Q2} 
\begin{flalign}
&Q^1 := \quad(w_3, y_3) = (0,0),\label{eq:Q1} \\
&Q^3 := \quad(w_3,y_3) = \left(0, \frac{\alpha}{\xi}\right).\label{eq:Q2} 
\end{flalign}
\end{subequations}
 The point $Q^1$ is a stable improper node with the double eigenvalue $-\alpha$ and a single eigenvector $(0,1)^T$. The point $Q^3$  has one unstable direction   $(0,1)^T$ due to the positive eigenvalue $\alpha$ and a  center direction  $(\alpha/(1+\alpha),1)^T$ due to a zero eigenvalue. 
Notice that for $\alpha=\xi$ then  $Q^3=(0,1)$.  

\begin{lemma}\label{lem:forward_ham}
There exists a unique center-stable manifold $W^{c,s}$ at the point $Q^3$. This manifold is smooth in $\alpha$. For $\alpha=\xi$ the set  $H=1$ coincides with $W^{c,s}$.
\end{lemma}

\begin{proof}
For $\alpha=\xi$ we rewrite the Hamiltonian \eqref{eq:H_hamilt} in chart $k_3$ and insert the condition $H=1$ to obtain the implicit equation:
\be\label{eq:Hamilt_Chart3} \xi(y_3-1) + w_3 (\xi+1) - \xi w_3 \rme^{-\frac{1}{w_3}}=0 ,\ee
then $w_3\to0$ gives $y_3 \to 1$ that is the point $Q^3$. As a consequence $Q^3$ has a saddle-like behaviour with an unique center-stable manifold $W^{c,s}$ tangent to $(\alpha/(1+\alpha),1)^T$. This invariant manifold  $W^{c,s}$ is smooth in $\alpha$ and therefore it preserves its features for small variations of $\alpha$ from $\alpha=\xi$.
\qed \end{proof}

\begin{remark}
With respect to $t_3$ the points within $W^{c,s}$ decay algebraically to $Q^3$, while the decay towards the stable node $Q^1$ is exponential. Using \eqref{eq:timet3}   it then  follows that all these points  reach $w_3=0$ in finite time with respect to the original slow time $t$. This is a formal proof of the finite time blow-up of solutions of \eqref{eq:reduced_problem} for $\alpha>\xi$ that was also observed by  \citename{gu1984a}  \citeyear{gu1984a} and by \citename{pomeau2011critical} \citeyear{pomeau2011critical}.
\end{remark}

\subsection{Chart $k_1$} \label{subsec:K1}
We insert \eqref{eq:chart_kappa1def} into the reduced problem \eqref{eq:reduced_problem_2} to obtain the dynamics in chart $k_1$:
 \be\label{eq:compactified_reduced_K1} 
\begin{aligned}
\dot{w} &= w^2(1-\rme^{\frac{z}{w}}),\\
\dot{z} &= w(\xi+z)(1-\rme^{\frac{z}{w}}) + \rme^{\frac{z}{w}}(\alpha z - \xi),
\end{aligned} \ee
where we have dropped the subscript  for the sake of readability.
We observe that the exponential term in \eqref{eq:compactified_reduced_K1} is not well defined in the origin.  For this reason we introduce the  blow-up transformation:
\be \label{eq:blowup_compactification} 
w = \bar{r}\bar{\omega}, \qquad z = \bar{r}\bar{\zeta},\ee
where $(\bar{\omega},\bar{\zeta})\in S^1 = \{ (\bar{\omega},\bar{\zeta}): \bar{\omega}^2 + \bar{\zeta}^2 = 1\}$ and $\bar{r}  \geq0$.
We consider the following  charts:
\begin{subequations}
\label{eq:charts_compactification}
 \begin{flalign} 
& \kappa_1: \quad w = r_1 \omega_1, \quad z = r_1,\label{eq:kappa1} \\
& \kappa_2: \quad w = r_2 , \qquad z = r_2 \zeta_2,\label{eq:kappa2}\\ 
& \kappa_3: \quad w = r_3 \omega_3, \quad z = -r_3.\label{eq:kappa3}
\end{flalign}
\end{subequations}
Next we perform an analysis of the blown-up vector field and the main results are summarized in Figure  \ref{fig:proof_bifurcation}.
\paragraph{Chart $\kappa_1$} We insert condition \eqref{eq:kappa1} into system   \eqref{eq:compactified_reduced_K1} and  divide the right hand side by $\exp(1/\omega_1)/r_1$  to get the de-singularized dynamics in chart $\kappa_1$: 
\be \label{eq:kap1_reduced1}
\begin{aligned}
\dot{\omega}_1 &= \omega_1(\xi- \alpha r_1) + r_1\omega_1^2 \xi\left(1 - \rme^{-\frac{1}{\omega_1}} \right),\\
\dot{r}_1 &= -r_1( \xi- \alpha r_1) - r_1^2 \omega_1(\xi + r_1)\left(1 - \rme^{-\frac{1}{\omega_1}} \right) .
\end{aligned}\ee
System \eqref{eq:kap1_reduced1} has  one fixed point in $(\omega_1,r_1) = (0,\xi/\alpha)$ that corresponds to the point $Q^3$ introduced in \eqref{eq:Q2}. 
Furthermore system  \eqref{eq:kap1_reduced1}  has a second fixed point  in $O_1:=  (\omega_1,r_1) = (0,0)$ with eigenvalues $\xi,\,\,-\xi$ and corresponding eigenvectors $(1,0)^T$ and $(0,1)^T$. 
Both the eigendirections of $O_1$ are invariant and we denote  by  $\gamma_1$ the heteroclinic connection between $Q^3$ and $O_1$ along the $r_1$-axis.\\
The initial condition $p_{1,\text{in}}$ on $W^{c,s}$ with  $\omega_1 = \delta>0$ is connected through the stable and the unstable manifolds of $O_1$ to the point $p_{1,\text{out}}:= (\omega_1,r_1) =  (\delta^{-1},0)$ as shown in Figure \ref{fig:reduced_kappa_1}.
\paragraph{Chart $\kappa_2$} We insert the transformation \eqref{eq:kappa2} into \eqref{eq:compactified_reduced_K1} and  divide the right hand side by $ \exp(\zeta_2)/r_2$ to obtain the de-singularized vector field.
In this chart there are no fixed points, yet the line  $r_2 =0$ is invariant and $\zeta_2$ decreases monotonically   along it. The orbit entering  from chart $\kappa_1$ has the initial condition $ p_{2,\text{in}}:= \kappa_{12}(p_{1,\text{out}})= \,\, (\zeta_2,r_2)= (\delta,0)$
 that lies  on the invariant line $r_2=0$.  Thus from $p_{2,\text{in}}$ we continue to the point $p_{2,\text{out}}:= (\zeta_2,r_2)= (-\delta^{-1},0)$, as shown in Figure \ref{fig:reduced_kappa_2}.
\paragraph{Chart $\kappa_3$} We introduce condition \eqref{eq:kappa3} into the vector field \eqref{eq:compactified_reduced_K1} and  divide by $w_3$ to obtain the de-singularized dynamics in chart $\kappa_3$: 
\be \label{eq:kap3_reduced1}
\begin{aligned}
\dot{\omega}_3 &= (\xi-r_3)(1-\rme^{-\frac{1}{\omega_3}})+ r_3\omega_3(1-\rme^{-\frac{1}{\omega_3}}) + \frac{\rme^{-\frac{1}{\omega_3}}}{r_3}(\alpha r_3 + \xi) ,\\
\dot{r}_3 &=  -r_3(\xi-r_3)(1-\rme^{-\frac{1}{\omega_3}}) - \frac{\rme^{-\frac{1}{\omega_3}}}{\omega_3} (\alpha r_3 + \xi).
\end{aligned}\ee
 System \eqref{eq:kap3_reduced1} has an unstable improper node in:  
\be \label{eq:Q^4} Q^7:= \quad (\omega_3,r_3) = (0,\xi),\ee
with double eigenvalue $\xi$ and  single eigenvector $(1,0)^T$.
For $w_3=r_3=0$ the quantity $\rme^{-1/\omega_3}/r_3$ in  \eqref{eq:kap3_reduced1} is not well defined.  We deal with this singularity  by first multiplying the right hand side of the vector field by $r_3\omega_3$:
\be \label{eq:kap3_reduced2}
\begin{aligned}
\dot{\omega}_3 &= r_3\omega_3(\xi-r_3)(1-\rme^{-\frac{1}{\omega_3}})+ r_3^2\omega_3^2(1-\rme^{-\frac{1}{\omega_3}}) + \omega_3\rme^{-\frac{1}{\omega_3}}(\alpha r_3 + \xi) ,\\
\dot{r}_3 &=  -r_3^2\omega_3(\xi-r_3)(1-\rme^{-\frac{1}{\omega_3}}) - r_3\rme^{-\frac{1}{\omega_3}} (\alpha r_3 + \xi).
\end{aligned}\ee
Next we introduce the  blow-up transformation:
\be \label{eq:nonstd_blowup}
\omega_3= \rho, \qquad r_3 = \frac{\rme^{-1/\rho}}{\rho} \eta.
 \ee
We substitute \eqref{eq:nonstd_blowup} into \eqref{eq:kap3_reduced2} and we divide by $ \exp(-1/\rho)/\rho$ to obtain the de-singularized vector field: 
\be \label{eq:kappa3blownup} \begin{aligned}
\dot{\rho} &= \xi \rho^2(\eta-1)  + \Or\left( \frac{\eta}{\rho} \rme^{-1/\rho} \right) ,\\
\dot{\eta} &= -\eta \xi (\eta-1)+ \Or\left(\frac{\eta}{\rho} \rme^{-1/\rho} \right).
\end{aligned}\ee

\begin{remark}
The blow-up map \eqref{eq:nonstd_blowup} is non-standard, since it is not written as an algebraic expression in $\rho$. To the author's knowledge there is no former literature   treating blow-ups of the form \eqref{eq:nonstd_blowup} and in particular the approach of   \cite{Kristiansen2015a} does not treat this type of blow-ups. 
\end{remark}
System \eqref{eq:kappa3blownup} has two fixed points. The first fixed point  $O_3:= (\rho,\eta) = (0,0)$  has one unstable direction $(0,1)^T$ associated  with the eigenvalue $\xi$ and one center direction $(1,0)^T$ associated  with the zero eigenvalue.  The second fixed point is: 
\be \label{eq:Q3} 
Q^6:=\quad  (\rho,\eta) = (0, 1),\ee
and it has one stable direction $(0,1)^T$ associated with the eigenvalue $-\xi$ and one center direction  $(1,0)^T$ associated  with the zero eigenvalue.  The axis $\rho=0$ is invariant, thus there exists an heteroclinic connection along the $\eta$-axis between the points $O_3$ and $Q^6$ that we denote by $\gamma_3$, see Figure \ref{fig:reduced_kappa_3}. 
\begin{lemma}\label{lem:backwards_ham}
There exists a unique center-unstable manifold $W^{c,u}$ at the point $Q^6$ that is smooth in $\alpha$ and that contains solutions that decay algebraically to $Q^6$ backwards in time. For $\alpha=\xi$ the set $H=1$  coincides with $W^{c,u}$.
\end{lemma}
\begin{proof}
We rewrite the Hamiltonian \eqref{eq:H_hamilt} in the $(\rho,\eta)$ coordinates and  then insert the condition $H=1$ to obtain the implicit equation:
\be \label{eq:Hamilt_chart3} \frac{1}{\eta} -1 + \rme^{-\frac{1}{\rho}}\left(\frac{1}{\rho} + 1 + \frac{1}{\xi}\right) =0. \ee
Here $\rho\to0$ gives $\eta \to 1$. Therefore $Q^6$ has a saddle-like behaviour with a unique center-unstable manifold   $W^{c,u}$ that is tangent to $(1,0)^T$ in $Q^6$. The invariant manifold $W^{c,u}$  is smooth in $\alpha$ and it maintains the center-unstable properties for small variation of $\alpha$ from $\alpha=\xi$.
\qed\end{proof}
The orbit entering from chart $\kappa_2$ in the point $p_{3,\text{in}} := \kappa_{23}(p_{2,\text{out}}) = (\rho,\eta) = (\delta,0)$   is connected through the stable and the unstable manifolds of $O_3$ to the point $p_{3,\text{out}}$ on $W^{c,u}$ with  $\omega_3 = \delta$ as shown in Figure \ref{fig:reduced_kappa_3}.

\begin{remark}
We observe that the singularity at the origin of chart $k_1$  \eqref{eq:compactified_reduced_K1}, upon blow-ups   \eqref{eq:blowup_compactification} and \eqref{eq:nonstd_blowup}, has turned into three hyperbolic fixed points $O_1, O_3$ and $Q^6$. After the blow-down we obtain the singular structure  depicted in Figure \ref{fig:reduced_all_infinity}.
\end{remark}
\begin{figure}[t!]\centering
 \begin{subfigure}{.28\textwidth}
  \centering 
  \caption{ }\label{fig:reduced_kappa_1}
\includegraphics[width=\linewidth]{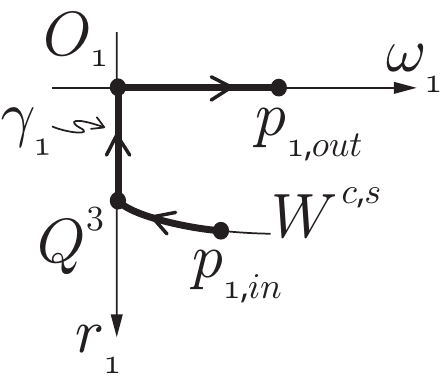}
\end{subfigure}\qquad
 \begin{subfigure}{.35\textwidth}
  \centering 
  \vspace{-1.4cm}
  \caption{ }\label{fig:reduced_kappa_2}
\includegraphics[width=\linewidth]{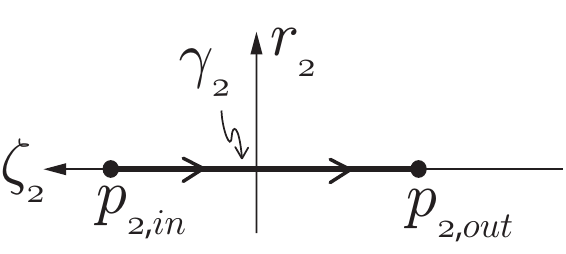}
\end{subfigure}\qquad
\begin{subfigure}{.28\textwidth}
  \centering 
  \caption{ }\label{fig:reduced_kappa_3}
\includegraphics[width=\linewidth]{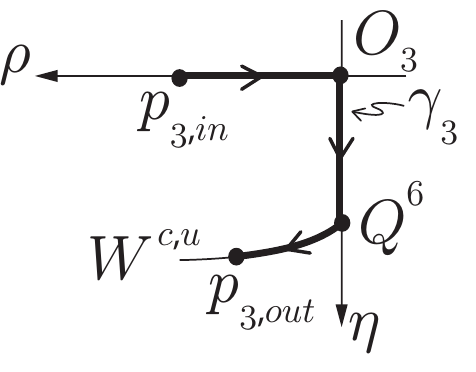}
\end{subfigure}
\begin{subfigure}{0.6\textwidth}
  \centering 
  \caption{}\label{fig:reduced_all_infinity}
\includegraphics[width=0.6\linewidth]{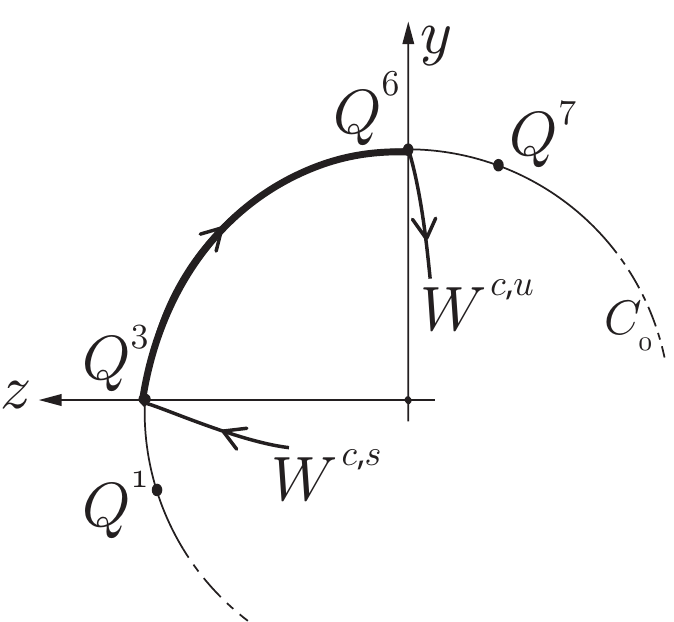}
\end{subfigure}
\caption{Blow-up of \eqref{eq:compactified_reduced_K1}  in chart $k_1$. \subref{fig:reduced_kappa_1}, \subref{fig:reduced_kappa_2} and \subref{fig:reduced_kappa_3} represent charts $\kappa_1,\kappa_2$ and $\kappa_3$ respectively. In \subref{fig:reduced_all_infinity}: behaviour at infinity after the blow-down.  }\label{fig:proof_bifurcation}
\end{figure}

\subsection{The reduced problem on $\mathcal{S}^{2,+}$ }\label{sec:infinity_to_reduced}
The previous analysis has described the phase space of \eqref{eq:reduced_problem} near infinity.  In the following we analyse the interaction of the unbounded solutions of the reduced problem \eqref{eq:reduced_problem} with the fixed points $Q^{1,3,6,7}$ for variations of the parameter $\alpha$.                                                                                                                                                                                                                                                                                                                                                                                               
 We follow the Melnikov-type approach of \citename{Chow94}  \citeyear{Chow94}, to describe how the closed orbits of the Hamiltonian system \eqref{eq:hamiltonian_sys} break up near $\alpha=\xi$. \\
When $\alpha=\xi$  any bounded trajectory of \eqref{eq:hamiltonian_sys}  with  $H=h, \,\,h\in(0,1)$,  intersects the $y$-axis in the two points $D,d$ that correspond to the two real roots of the Lambert equation \eqref{eq:lambert}. We denote by $D$  the root with $y>0$ while we denote by $d$  the one with $y<0$, see Figure  \ref{fig:perturbation_hamiltonian}.\\
For $\alpha-\xi$ small, we compute the forward and backwards orbits  $\gamma^+(t)$ and $\gamma^-(t)$ respectively emanating from  $D$. The transversality condition \eqref{eq:transversality} assures that  $\gamma^+(t)$ and $\gamma^-(t)$  cross the $y$-axis for the first time  in the points $d^+$ and $d^-$ respectively.  Hence we define the distance function:
\be \label{eq:distance_perturbation} 
\begin{aligned}
\Delta(\alpha)&=H(d^+)-H(d^-) ,\\
&= \int_0^{T^+} \dot{H}(\gamma^+(t))\,\rmd t+ \int_{T^-}^0 \dot{H}(\gamma^-(t))\,\rmd t,\\
&= \int_0^{T^+} \nabla{H(h)}\cdot f_0(y,z;\alpha)\,\rmd t+ \int_{T^-}^0  \nabla{H(h)}\cdot f_0(y,z;\alpha)\,\rmd t,
\end{aligned}\ee
where $T^\pm=T^\pm(\alpha)\gtrless 0$ is the flow-time between $D$ and $d^+$  and between  $D$ and  $d^-$ respectively. 
 \begin{figure}\centering
 \begin{subfigure}{.4\textwidth}
  \centering 
  \caption{ }\label{fig:perturbation_hamiltonian}
\includegraphics[width=0.87\linewidth]{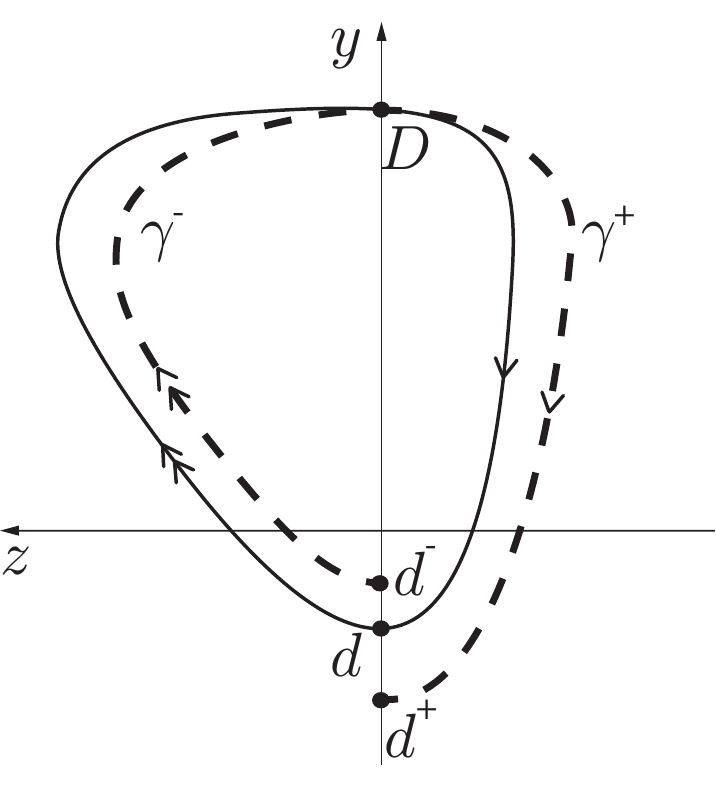}
\end{subfigure}
\quad
\begin{subfigure}{.4\textwidth}
  \centering 
  \caption{  }\label{fig:perturbation_h1}
\includegraphics[width=\linewidth]{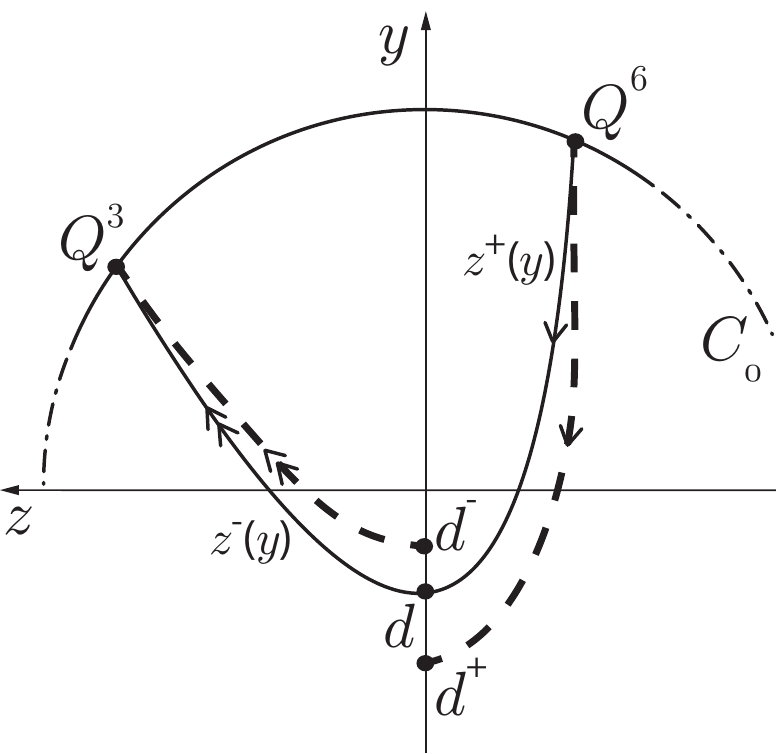}
\end{subfigure}
\caption{Perturbation of the Hamiltonian solutions for $\alpha-\xi$ small. In \subref{fig:perturbation_hamiltonian}:  closed orbit with $0<H<1$. In \subref{fig:perturbation_h1}: heteroclinic connection for $H=1$. }\label{fig:perturbation}
\end{figure}
We Taylor expand \eqref{eq:distance_perturbation} around $\alpha=\xi$:
\be \label{eq:taylor_expand} 
\Delta(\alpha) = (\alpha-\xi)\Delta_{\alpha}(h) + {\Or((\alpha-\xi)^2),}
\ee
with the quantity $\Delta_{\alpha}(h)$ defined as:
\be\label{eq:dalpha} 
\begin{aligned}
  \Delta_{\alpha}(h) &= \int_{T_h^-}^{T_h^+} \nabla H(h)\cdot \frac{\partial f_0}{\partial \alpha}(y,z; \xi)\, \rmd t  \\
  &= \int_{T_h^-,}^{T_h^+}   \xi \rme^{-\xi y} z(\rme^{z}-1)  \, \rmd t.
  \end{aligned}\ee
In \eqref{eq:dalpha} we have denoted with $(y,z)(t)$   the solution of \eqref{eq:hamiltonian_sys} for $H=h$ and $\alpha=\xi$. The times $T_h^\pm=T_h^\pm(\xi)$ are the forward and backwards times from $D$ to $d$.  The integrand of \eqref{eq:dalpha} is always positive for $z\neq0$ and therefore $\Delta_\alpha(h)$ is positive for any $h\in(0,1)$.
We conclude from \eqref{eq:taylor_expand}   that the forward flow $\gamma^+(t)$ spirals outwards for $\alpha>\xi$ while it spirals inwards for $\alpha<\xi$, in agreement with Figure \ref{fig:bifurcation_alpha}. \\
We now extend the analysis above to the case of $H=1$. In this case the points $d^{+}$ and $d^{-}$ are the intersections of $W^{c,u}$ and $W^{c,s}$ with the $y$-axis respectively, see Figure \ref{fig:perturbation_h1}. From the analysis above we know that   $W^{c,s}$ and $ W^{c,u}$ depend smoothly on $\alpha$.
\begin{lemma}\label{lem:ham_compact}
For $\alpha=\xi$ there is a unique heteroclinic connection between $Q^3$ and $Q^6$ on $C_0$. This connection is through the manifolds $W^{c,s}$ and $W^{c,u}$ and it corresponds to the set $H=1$ in  \eqref{eq:H_hamilt}.  This set can be written as the union of two graphs  $z=z^\pm(y)$ (see Figure \ref{fig:perturbation_h1}) with $y\geq-1/\xi$ so that $z^-(y)\,\, \left(z^+(y) \text{ resp.}\right)$ approaches $Q^3 \,\,\left(Q^6\right)$ as $z^- = \Or(y)\,\, \left(z^+ = \Or\left(\ln(y)\right)\, \right)$ for $y\to \infty$.
 \end{lemma} 
 \begin{proof}
 We rewrite the trajectory $H=1$ as the graphs $z=z^\pm(y)$ for $y \geq -1/\xi$. The behaviour in forward time follows by considering the point $p_{1,\text{in}}$ in condition \eqref{eq:Hamilt_Chart3} and blowing it down to the original variables $(y,z)$. Similarly for the  behaviour in backwards time by considering $p_{3,\text{out}}$ in condition \eqref{eq:Hamilt_chart3}. 
\qed \end{proof}
Figure \ref{fig:alpha_equal_xi} follows from Lemma \ref{lem:ham_compact}.  When $\alpha=\xi$ the manifolds $W^{c,s}$ and  $W^{c,u}$  cross the $y$-axis in the point $d:= (y,z)=(-1/\xi,0)$.
 We define the distance function $\Delta(\alpha)$ as in \eqref{eq:distance_perturbation}, we Taylor expand it around $\alpha=\xi$ as in \eqref{eq:taylor_expand} and we define $\Delta_\alpha(1)$ as in \eqref{eq:dalpha}.  
  Since the integrand of \eref{eq:dalpha} is positive for $H=1$ we just need to show that  the improper integral \eqref{eq:dalpha} exists. 
From the reduced problem \eqref{eq:reduced_problem} we observe that $\dot{y} =\rme^{z}-1$, thus we rewrite \eqref{eq:dalpha}  with respect to $y$ as:
\be \label{eq:integral_h1_parametrized} \Delta_{\alpha}(1)= \int_{-1/\xi}^{+\infty}   \xi \rme^{-\xi y} z^{-}(y)\,\rmd y  - \int_{-1/\xi}^{+\infty}   \xi \rme^{-\xi y} z^{+}(y)\,\rmd y .
\ee
Recall from Lemma \ref{lem:ham_compact}  that $z^-(y)$ is asymptotically linear in $y$ for $y\to\infty$, while $z^+(y)$ decreases logarithmically  with respect to $y$. The expression  \eqref{eq:integral_h1_parametrized} therefore exists because of the exponential decay of  the factor $\exp({-\xi y}) $ and  furthermore it is positive.
We remark that $\Delta_\alpha(h)$ in \eqref{eq:dalpha} converges   to $\Delta_\alpha(1)$  for $h\to1$, since the orbit segment on $C_{0,\infty}$ does not give any contribution to \eqref{eq:integral_h1_parametrized}.\\

\noindent Now we finish the proof of Proposition \ref{theo:bifurcation_alpha} by considering $\alpha$ as in \eqref{eq:alpha_interval}. When $\alpha<\xi$ the set $W^{c,u}$ contracts to the origin, because $\Delta(\alpha)<0$ in \eqref{eq:taylor_expand}. Furthermore the set $W^{c,s}$ is backwards asymptotic to $Q^7$ and acts as a separator between the basin of attraction of the origin and the basin of attraction of $Q^1$. A similar argument covers the case $\alpha>\xi$. 
This concludes the proof of  Proposition \ref{theo:bifurcation_alpha} and justifies  Figures \ref{fig:alpha_less_xi} and \ref{fig:alpha_larger_xi}. 
Therefore no periodic orbit  exists on $C_0$ for $\alpha>\xi$ and $\vare=0$.  
 
\section{Analysis of the perturbed problem for $\vare>0$}\label{sec:perturbed_reduced}
Consider the original problem \eqref{eq:3Dproblem} and $0<\mu<1$ small but fixed. Then the compact manifold:
\be\label{eq:compact_manifold} 
S_0 = \{ (x,y,z)\in C_0 \lvert \quad 0\leq H(y,z)\leq1-\mu\},\ee
is normally hyperbolic for $\vare=0$.
Therefore Fenichel's theory  guarantees that for $\vare$ sufficiently small there exists a locally invariant manifold $S_\vare$ that is $\Or(\vare)$-close to $S_0$ and is diffeomorphic to it. Moreover the flow on $S_\vare$ converges to the flow of the reduced problem \eqref{eq:reduced_problem} for $\vare\to0$.
A computation shows that $S_\vare$ at first order is:
\[ z = -(x+\xi y) + \vare\xi \rme^{-2(x+\xi y)}\left( \alpha(x+\xi y) + \xi(y+1) - \xi \rme^{x+\xi y}\right)+ \Or(\vare^2),\]
hence we have the following vector field $f_\vare(y,z;\alpha,\vare)$  on  $S_\vare$:
\be\label{eq:perturbed_vf} f_\vare(y,z;\alpha,\vare) := \begin{cases}
\dot{y} &= \rme^{z}-1 - \vare\xi\chi\rme^{2z} +\Or(\vare^2),\\
\dot{z} &= \chi-\vare\xi\chi\rme^{2z}(\alpha z- \xi y +\alpha-\xi+1)+\Or(\vare^2),
\end{cases}\ee
with $\chi(y,z)=\alpha z\rme^{z}-\xi y\rme^{z}  - \xi\rme^{z} +\xi$.

\begin{proposition}\label{prop:hopf_perturbed}
Consider the compact manifold $S_0$ defined in \eqref{eq:compact_manifold}. Then $S_0$ perturbs to a locally invariant slow manifold $S_\vare$ for $0<\vare\ll1$. On $S_\vare$ the origin of \eqref{eq:perturbed_vf} undergoes a supercritical Hopf bifurcation for:
\be \label{eq:alpha_hopf} \alpha= \alpha_H:=\xi - \vare \xi^2 + \Or(\vare^2),\ee
with a negative first Lyapunov coefficient: 
\be\label{eq:normal_form_coeff}  a = -\frac{1}{8} \vare\xi^3(1+\xi) + \Or(\vare^2)<0.\ee
Therefore for $\alpha\in\left(\alpha_H, \alpha_H+c(\mu)\vare\right)$ with $c(\mu)$ sufficiently small, there  exists a family of locally unique attracting limit cycles with amplitude of order $\Or\left(\sqrt{-(\alpha-\alpha_H)/a}\right)$.
\end{proposition}
 The proof of Proposition \ref{prop:hopf_perturbed}  follows from straightforward computations. We remark that since \eqref{eq:normal_form_coeff} is proportional to $\vare$, it follows that the results of Proposition \ref{prop:hopf_perturbed} are valid only for a very small interval of $ \alpha$ around $ \alpha_H$.
We use  the  analysis of section \ref{sec:infinity_to_reduced}  to extend the small limit cycles  of Proposition \ref{prop:hopf_perturbed} into   larger ones.

\begin{proposition}\label{cor:extension_epsilon}
Consider the slow manifold $S_\vare$ of Proposition  \ref{prop:hopf_perturbed}.   On $S_\vare$ there exists a  family of closed periodic orbits for 
\be\label{eq:alpha_melnikov}  \alpha=\alpha_M(h):=\xi - \vare \frac{\Delta_\vare(h)}{\Delta_\alpha(h)} + \Or(\vare^2),\ee
where $h\in[c_1(\mu),1-c_2(\mu)]$ with $(c_1,c_2)(\mu)$ small. The quantity $\mathit{\Delta}_\vare(h)$ is defined as:
\be\label{eq:delta_epsilon}
\Delta_{\vare}(h) = \int_{T_h^-}^{T_h^+} \nabla H(h)\cdot \frac{\partial f_\vare}{\partial \vare}(y,z; \xi,0)\, \rmd t,\ee
while $\Delta_\alpha(h)>0$ was  defined in \eqref{eq:dalpha}.
\end{proposition}

\begin{proof}
By Fenichel's theorem we know that the flow on $S_\vare$ converges to the flow of the reduced problem \eqref{eq:reduced_problem} for $\vare\to0$. Therefore we can define the distance function $ \Delta(\alpha,\vare)$ similarly to \eqref{eq:distance_perturbation} whose Taylor expansion around $\alpha=\xi$ and $\vare=0$ is: 
\be \label{eq:delta_perturbed_taylor} 
\Delta(\alpha,\vare) =   (\alpha-\xi) \Delta_{\alpha}(h) + \vare\Delta_\vare (h) + \Or((\alpha-\xi+\vare)^2),\ee
with $\Delta_\alpha (h)$ and $\Delta_{\vare}(h)$ defined in \eqref{eq:dalpha} and \eqref{eq:delta_epsilon} respectively. The integrand of $\Delta_\alpha(h)$ is strictly positive for all $h\in(0,1)$, therefore we can apply the implicit function theorem to \eqref{eq:delta_perturbed_taylor} for $\Delta(\alpha,\vare) =0$ and obtain the result  \eqref{eq:alpha_melnikov}.
\qed\end{proof}
In Figure \ref{fig:Melnikov_numerics} we show a numerical computation of the leading order coefficient in \eqref{eq:alpha_melnikov}  for an interval of energies $H=h\in(0,0.6]$. No saddle-node bifurcations occur in this interval and hence the periodic orbits are all asymptotically stable. We expect a similar behaviour for larger values of $h$ but we did not manage to compute this due to the intrinsic slow-fast structure of the reduced problem. It might be possible to study the term $\Delta_\vare(h)/\Delta_\alpha(h)$ analytically by using the results of Lemma \ref{lem:ham_compact}   but the expressions are lengthy and we did not find an easy way.\\
The analysis above can  only explain the limit cycles that appear for $\alpha-\xi=\Or(\vare)$ and it does not justify the limit cycles of Figure \ref{fig:3dnumerics} that appear for larger values of $\alpha-\xi$.
For this reason we proceed to study the  full problem \eqref{eq:3Dproblem} at infinity,  introducing its compactification  through  the Poincar\'e sphere.

 \begin{figure}\centering
 \includegraphics[width=0.4\linewidth]{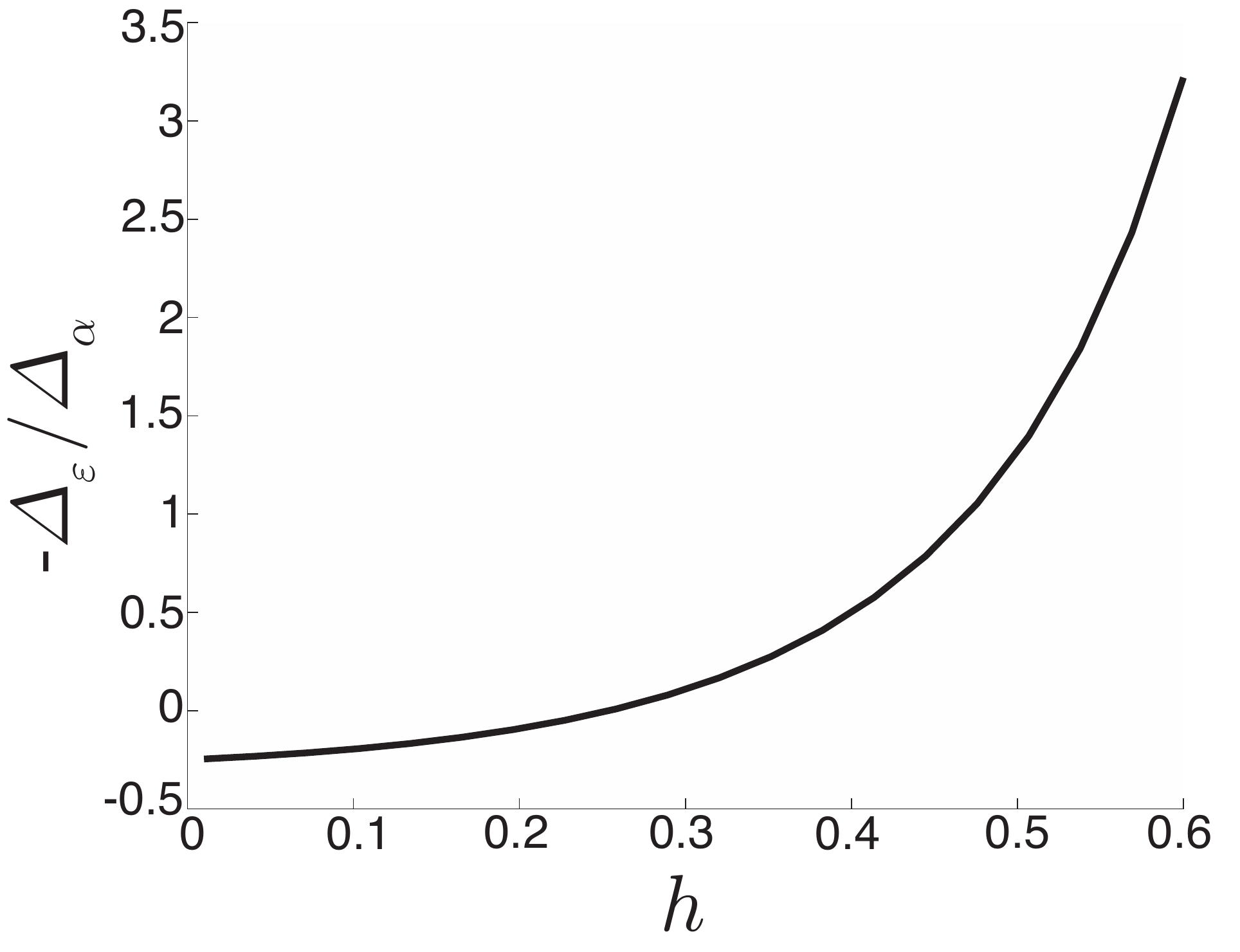}
\caption{Plot of the leading order coefficient in \eqref{eq:alpha_melnikov} for $\xi=0.5$ and $h\in(0,0.6]$. }\label{fig:Melnikov_numerics}
\end{figure}

\section{Statement of the main result}\label{sec:infinity}
 In this section we find a  connection at infinity between the points $Q^1$ and $Q^6$ (recall Proposition \ref{prop:fixed_point_infinity})  that will establish a return mechanism to $C_0$ of the unbounded solutions of \eqref{eq:fast_problem}  when  $\vare=0$ and $\alpha>\xi$. This mechanism will be the foundation for the existence of limit cycles when  $0<\vare\ll1$ and $\alpha-\xi\geq c>0$.  \\
Similar to section \ref{sec:compact_reduced}, we introduce a four-dimensional Poincar\'e sphere $\mathcal{S}^{3,+}$:
\be \label{eq:sphere_def_3d} 
\mathcal{S}^{3,+} := \left\{(X,Y,Z,W)\in \mathbb{R}^4 \bigr|\quad  X^2+Y^2+Z^2+W^2=1, \quad W\geq0\right\}.\ee
  The fast problem \eqref{eq:fast_problem} is interpreted as a directional chart ${K}_2$ on $\mathcal{S}^{3,+}$ defined for $W=1$:
\[ \fl {K}_2:= \mathcal{S}^{3,+} \cap \{W=1\}, \quad x_2 = \frac{X}{W}, y_2 = \frac{Y}{W}, z_2 = \frac{Z}{W},\]
therefore the vector field in chart $K_2$ is obtained by introducing the subscript in \eqref{eq:fast_problem}:
\be\label{eq:fast_problem_2}
\begin{aligned}
\dot{x}_2 &= -\vare \rme^{z_2}(x_2 +(1+\alpha) z_2),\\
\dot{y}_2 &= \vare \left(\rme^{z_2} -1\right),\\
\dot{z}_2 &= -\rme^{-z_2}\left(y_2 + \frac{x_2+z_2}{\xi} \right).
\end{aligned}\ee 
The points at infinity in ${K}_2$  correspond to  $W=0$ which is a sphere ${S}^2$.
We introduce the two  directional charts: 
\begin{subequations}
\label{eq:new_charts_3d}
 \begin{flalign}
 &{K}_3:= \mathcal{S}^{3,+} \cap \{Z=1\}, \quad  x_3 = \frac{X}{Z}, y_3 = \frac{Y}{Z}, w_3 = \frac{W}{Z},\label{eq:R&S_kappa3_3d} \\
&{K}_1:= \mathcal{S}^{3,+} \cap \{Y=1\}, \quad  x_1 = \frac{X}{Y}, z_1 = \frac{Z}{Y}, w_1 = \frac{W}{Y}.\label{eq:R&S_kappa1_3d}
\end{flalign}
\end{subequations}
We have the following transformations between the charts: 
\begin{subequations}
\label{eq:change_charts_3d}
 \begin{flalign}
&{K}_{23}:\quad  w_3= z_2^{-1}, \quad\, x_3 = x_2 z_2^{-1},\quad  y_3 =y_2 z_2^{-1},\label{eq:R&S_kappa23_3d} \\
&{K}_{21}:\quad  w_1 = y_2^{-1}, \quad\, x_1 = x_2 y_2^{-1},  \quad z_1 = z_2y_2^{-1},\label{eq:R&S_kappa21_3d} \\
&{K}_{31}:\quad w_1 = w_3 y_3^{-1}, \, x_1 = x_3 y_3^{-1}, \quad z_1= y_3^{-1},\label{eq:R&S_kappa31_3d} 
\end{flalign}
\end{subequations}
that  are defined for $z_2>0$, $y_2>0$ and $y_3>0$ respectively. The inverse transformations are defined similarly. The three points $Q^1,Q^3\in K_3$ and $Q^6\in K_1$: 
\begin{subequations}
\label{eq:Q1Q2Q3}
 \begin{flalign}
&Q^1 :=\quad  (x_3,y_3,w_3)=(-1,0,0),\label{eq:Q1_3D} \\
&Q^3 :=\quad  (x_3,y_3,w_3)=\left(-1-\alpha,\frac{\alpha}{\xi},0\right),\label{eq:Q2_3D} \\
&Q^6 :=\quad  (x_1,z_1,w_1)=\left(-\xi,0,0\right),\label{eq:Q3_3D} 
\end{flalign}
\end{subequations}
introduced in Proposition \ref{prop:fixed_point_infinity} and the three points $Q^2,Q^4\in K_3$ and $Q^5\in K_1$:
\begin{subequations}
\label{eq:QaQbQc}
 \begin{flalign}
&Q^2 :=\quad  (x_3,y_3,w_3)=(-1-\alpha,0,0),\label{eq:Qa} \\
&Q^4 :=\quad  (x_3,y_3,w_3)=\left(-1-\alpha,\frac{2\alpha}{\xi},0\right),\label{eq:Qb} \\
&Q^5 :=\quad  (x_1,z_1,w_1)=\left(-\frac{\xi}{2\alpha}(1+\alpha),\frac{\xi}{2\alpha}(1-\alpha),0\right),\label{eq:Qc} 
\end{flalign}
\end{subequations}
are going to play a role in the  following, together with the lines:
\begin{subequations}
\label{eq:L0C0infty}
 \begin{flalign}
 L_0 &:=\quad \left\{(x_3,y_3,w_3) \lvert \quad  x_3 + 1+ \alpha  =0,\,\, w_3 =0 \right\},\label{eq:L0_K3} \\
 C_{0,\infty} &:= \quad\left\{(x_3,y_3,w_3) \lvert \quad  x_3 + \xi y_3 + 1  =0,\,\, w_3 =0 \right\}.\label{eq:eq:C0infty_K3} 
\end{flalign}
\end{subequations}
Notice that the line $L_0$ corresponds to the intersection of the family of nullclines  \eqref{eq:L0} with  infinity through  $K_{23}$. We construct the following \textit{singular cycle}:

\begin{definition}
 Let  $\Gamma_0$ be the closed orbit consisting of the points $Q^{1,2,4,5,6}$ and of the union of the following sets:
\begin{itemize}
\item $\gamma^{1,2}$   connecting $Q^1$ with  $Q^2$. In    chart ${K}_3$ the segment $\gamma^{1,2}$  is:
\be \label{eq:gamma_1a}
\gamma^{1,2} := \{ (x_3,y_3,w_3)\in {K}_3 \lvert \quad x_3 \in (-1-\alpha, -1), y_3=0, w_3=0\}.
\ee
\item  $\gamma^{2,4} $  connecting  $Q^2$  with $Q^4$ along $L_0$.  In    chart ${K}_3$ the segment $\gamma^{2,4}$  is:
\be \label{eq:gamma_ab}
\gamma^{2,4} := \left\{ (x_3,y_3,w_3)\in {K}_3 \lvert \quad x_3 = -1-\alpha  , y_3\in \left(0,\frac{2\alpha}{\xi}\right), w_3=0\right\}.
\ee
\item   $\gamma^{4,5}$   connecting $Q^4$   with   $Q^5$. This segment is a fast fiber of \eqref{eq:layer_problem} and in chart ${K}_1$ the segment  $\gamma^{4,5}$  is:
{\small\be \label{eq:gamma_bc}
\gamma^{4,5} := \left\{ (x_1,z_1,w_1)\in {K}_1 \lvert \quad x_1 =  -\frac{\xi}{2\alpha}(1+\alpha)  , z_1\in\left(\frac{\xi}{2\alpha}(1-\alpha), \frac{\xi}{2\alpha}\right), w_1=0\right\}.
\ee}
\item   $\gamma^{5,6}$ connecting  $Q^5$   with  $Q^6$ on $C_{0,\infty}$.  In    chart ${K}_1$ the segment $\gamma^{5,6}$  is:
{\small \be \label{eq:gamma_c3}
\gamma^{5,6} := \left\{ (x_1,z_1,w_1)\in {K}_1 \lvert \quad x_1=-\xi-z_1, z_1 \in \left(0,\frac{\xi}{2\alpha}(1-\alpha)\right),   w_1=0\right\}.
\ee}
\item $W^{c,u}$ connecting $Q^6$ with $Q^1$ on the critical manifold $C_0$.
\end{itemize}
\end{definition}
In section  \ref{sec:proof} we identify $\Gamma_0$ using repeatedly the blow-up method on system \eqref{eq:fast_problem_2}.  Figure \ref{fig:sphere_infinity} shows $\Gamma_0$ and its different segments: \ref{fig:sphere_global}  displays the complete cycle while  \ref{fig:kappa3}  and \ref{fig:kappa1_new}  illustrate the portions of $\Gamma_0$ that are visible in the charts ${K}_3$ and ${K}_1$ respectively. 
 \begin{figure}[h!]\centering
 \begin{subfigure}{\textwidth}
  \centering 
    \caption{}\label{fig:sphere_global}
 \includegraphics[width=0.43\linewidth]{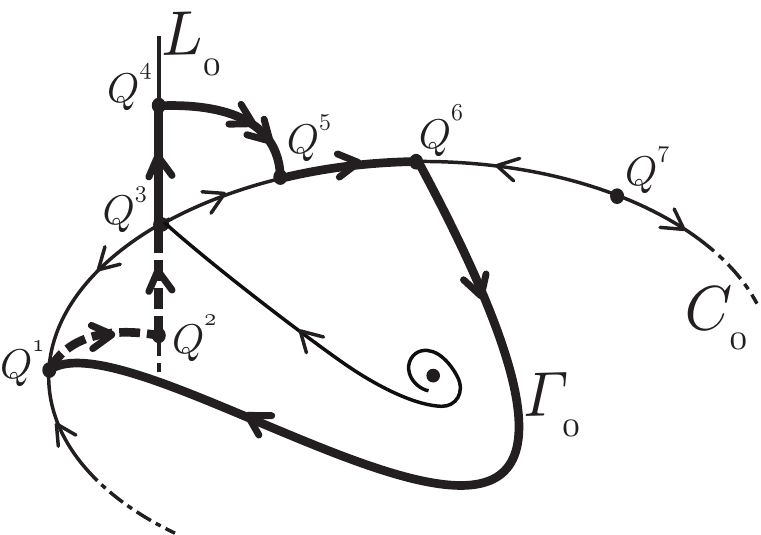}
\end{subfigure}
\quad
\begin{subfigure}{.48\textwidth}
  \centering 
  \caption{}\label{fig:kappa3}
\includegraphics[scale=0.8]{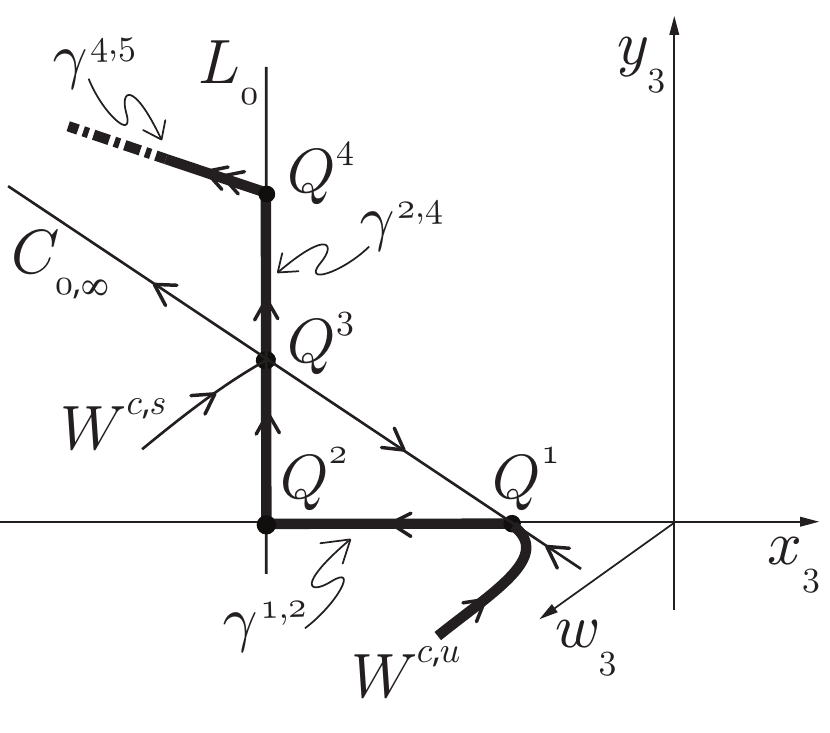}
\end{subfigure}
\begin{subfigure}{.48\textwidth}
  \centering 
  \caption{ }\label{fig:kappa1_new}
\includegraphics[height=5.7cm]{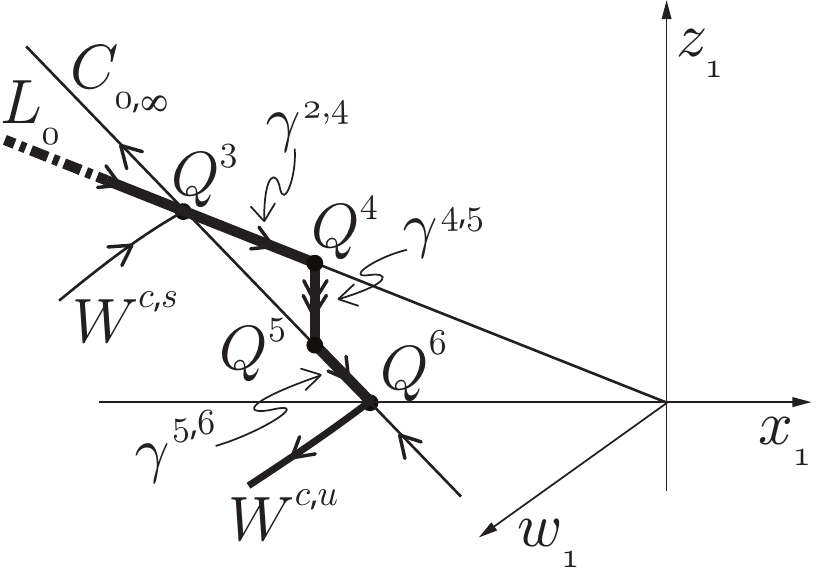}
\end{subfigure}
\caption{Schematisation of $\Gamma_0$ in \subref{fig:sphere_global}. In chart $K_3$ \subref{fig:kappa3} we see the segments $\gamma^{1,2}, \gamma^{2,4}$ and partially $\gamma^{4,5}$. In chart $K_1$ \subref{fig:kappa1_new} we see $\gamma^{4,5}, \gamma^{5,6}$ and partially $\gamma^{2,4}$. }\label{fig:sphere_infinity}
\end{figure}
$\Gamma_0$ plays an important role in our main result, since we conjecture it to be the candidate \textit{singular cycle}:
\begin{conjecture}\label{con:conjecture_thm}
Fix $\alpha>\xi$. Then for $0<\vare\ll1$ there exists an attracting limit cycle $\Gamma_{\vare }$ that converges to  the  \textnormal{singular  cycle} $\Gamma_0$ for $\vare\to0$.
\end{conjecture}
A rigorous proof of Conjecture \ref{con:conjecture_thm} requires an analysis both for $\vare=0$ and $0<\vare\ll1$.  In section \ref{sec:global_results} we outline a  procedure to prove the conjecture and we  leave the full details of the proof to a future manuscript.

\begin{remark}\label{rem:main_result}
Here we collect the results of sections \ref{sec:perturbed_reduced} and \ref{sec:infinity}. When $\vare=0$  and $\alpha=\xi$ then there exists a family of periodic solutions  on $\mathcal{S}^{3,+}$, corresponding to the Hamiltonian orbits with $H\in(0,1)$. For $\alpha>\xi$  only the cycle  $\Gamma_0$  persists. \\
When $0<\vare\ll1$ and $\alpha-\xi =\Or(\vare)$  there exists a limit cycle  resembling the bounded Hamiltonian orbits. For larger values of $\alpha-\xi$ we conjecture that the limit cycle tends to  $\Gamma_0$. Figure \ref{fig:conjecture_maxampl} shows the conjectured bifurcation diagram of the periodic orbits.
\begin{figure}[h!]
\centering
\includegraphics[scale=1.1]{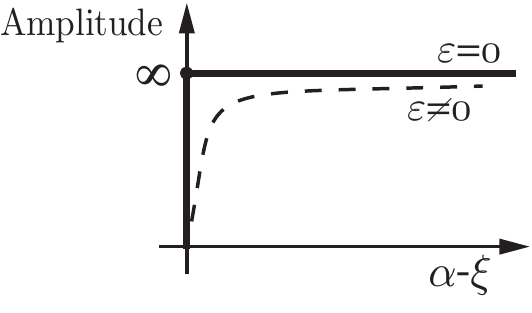}
\caption{ Conjectured bifurcation diagram of the limit cycles for $\vare\ll1$.}\label{fig:conjecture_maxampl}
\end{figure}
\end{remark}
 Figure \ref{fig:closedorbit}  shows some numerical simulations  supporting Conjecture \ref{con:conjecture_thm}: \ref{fig:num_closedorbit} illustrates the limit cycles $\Gamma_\vare$ for three different values of $\vare\in\{10^{-8},10^{-4},10^{-2}\}$ with $\alpha=0.9$ and $\xi=0.5$ while  \ref{fig:num_chartK3} and \ref{fig:numericsK1_3D} show the portions of $\Gamma_\vare$ that appear in the   charts $K_3$ and $K_1$ respectively. The amplitudes of the orbits increase for decreasing values of  the parameter $\vare$ and   both the  plane $C_0$ and the line $L_0$ play an important role. Close to the origin the dynamics evolves on $C_0$ while sufficiently far from the origin     $L_0$ becomes relevant. Indeed  in Figure \ref{fig:num_chartK3} we see that the solutions contract to $L_0$ following $\gamma^{1,2}$ and then they evolve    following $\gamma^{2,4}$.  
When the trajectories  are  close to $Q^4$ they follow $\gamma^{4,5}$ and   contract again  towards $C_0$ along a direction that  tends  to the fast fiber for $\vare\to0$, as we can see in Figure \ref{fig:numericsK1_3D}. 
 \begin{figure}[h!]\centering
 \begin{subfigure}{\textwidth}
  \centering 
  \caption{  }\label{fig:num_closedorbit}
\includegraphics[scale=0.64]{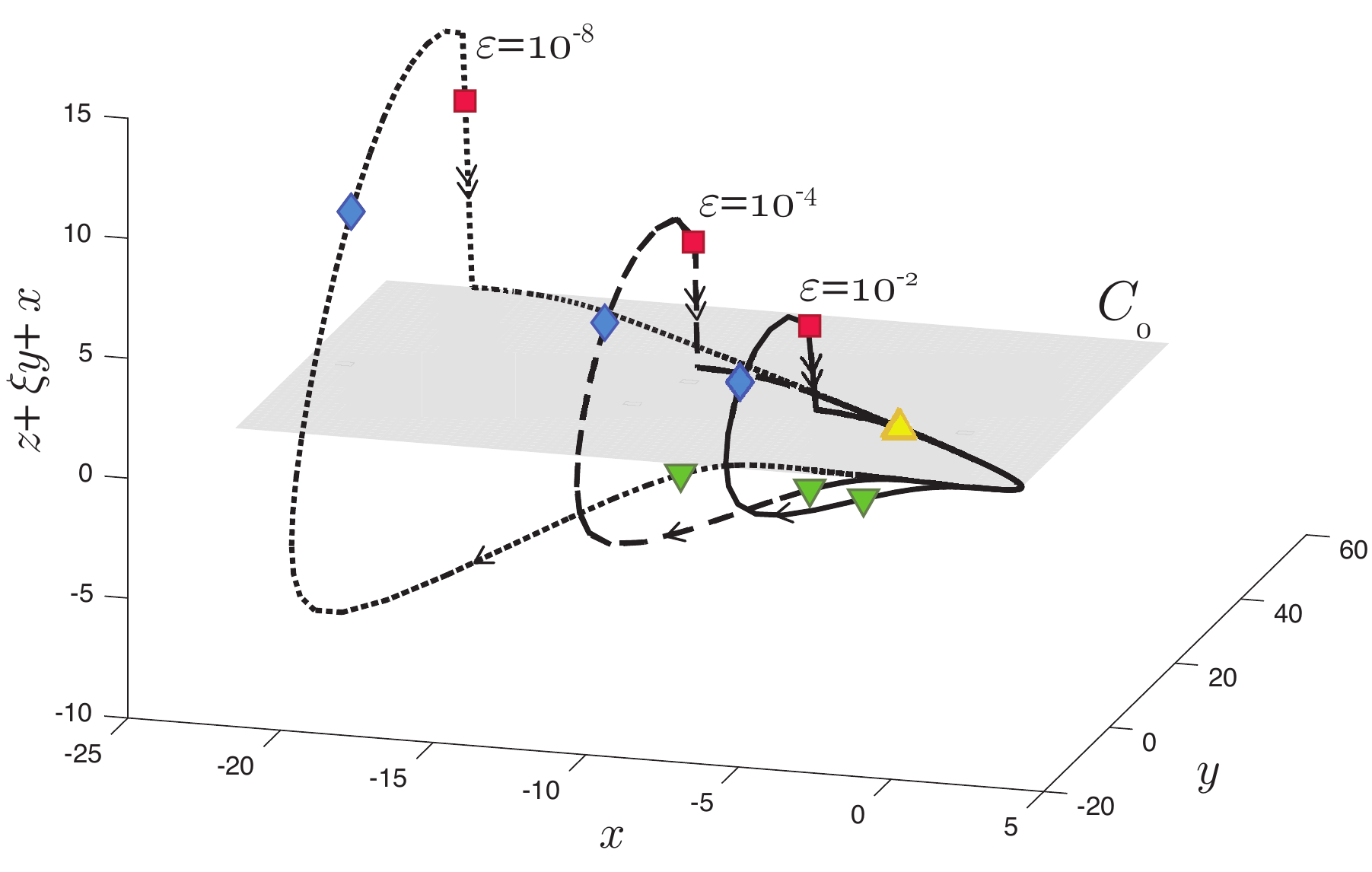}\end{subfigure}\\
    \begin{subfigure}{\textwidth}
  \centering 
  \caption{   }\label{fig:num_chartK3}
\includegraphics[scale=0.53]{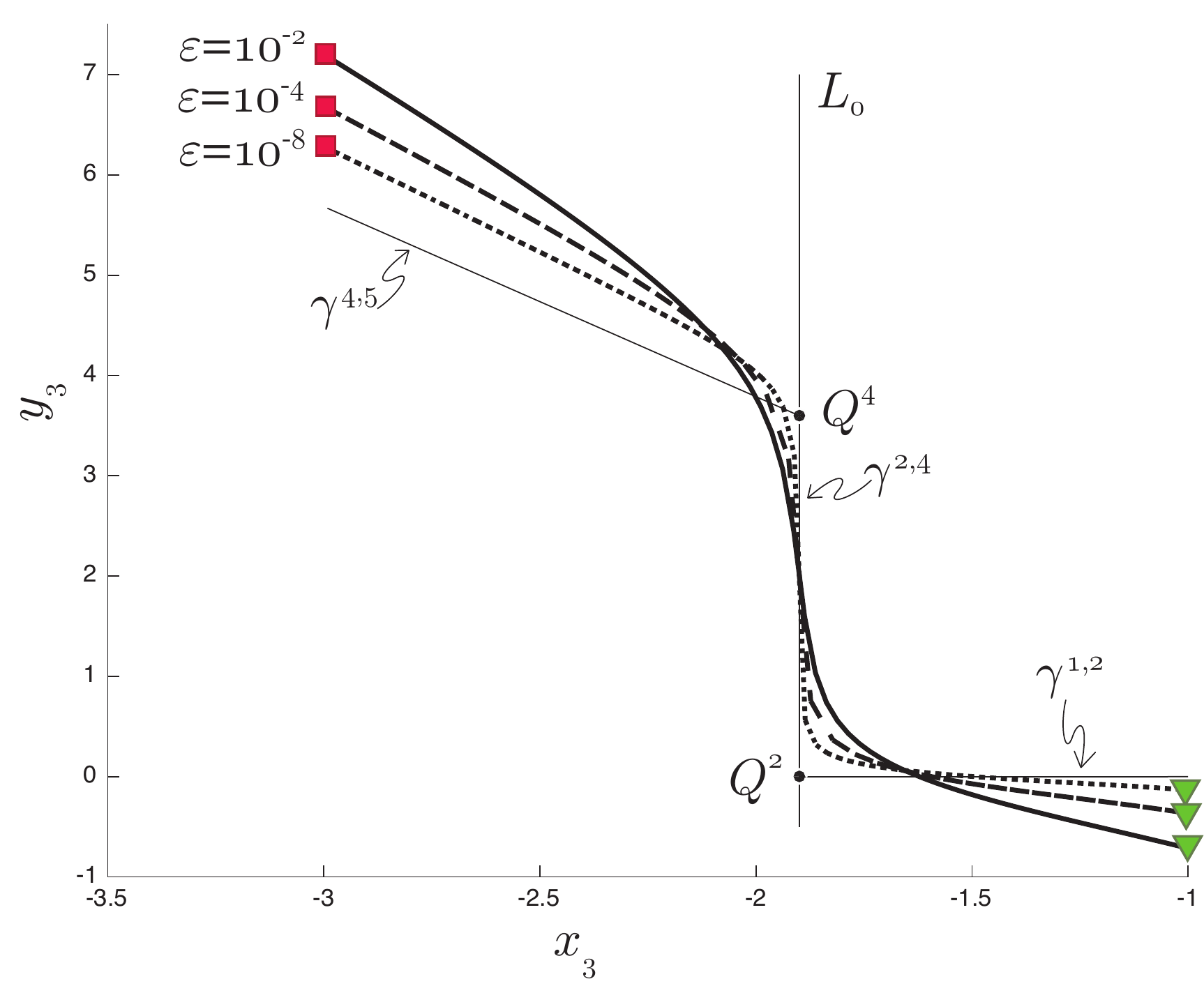}
\end{subfigure}
\end{figure}
\begin{figure}[h!] 
    \ContinuedFloat 
\begin{subfigure}{\textwidth}
  \centering 
  \caption{ }\label{fig:numericsK1_3D}
\includegraphics[scale=0.55]{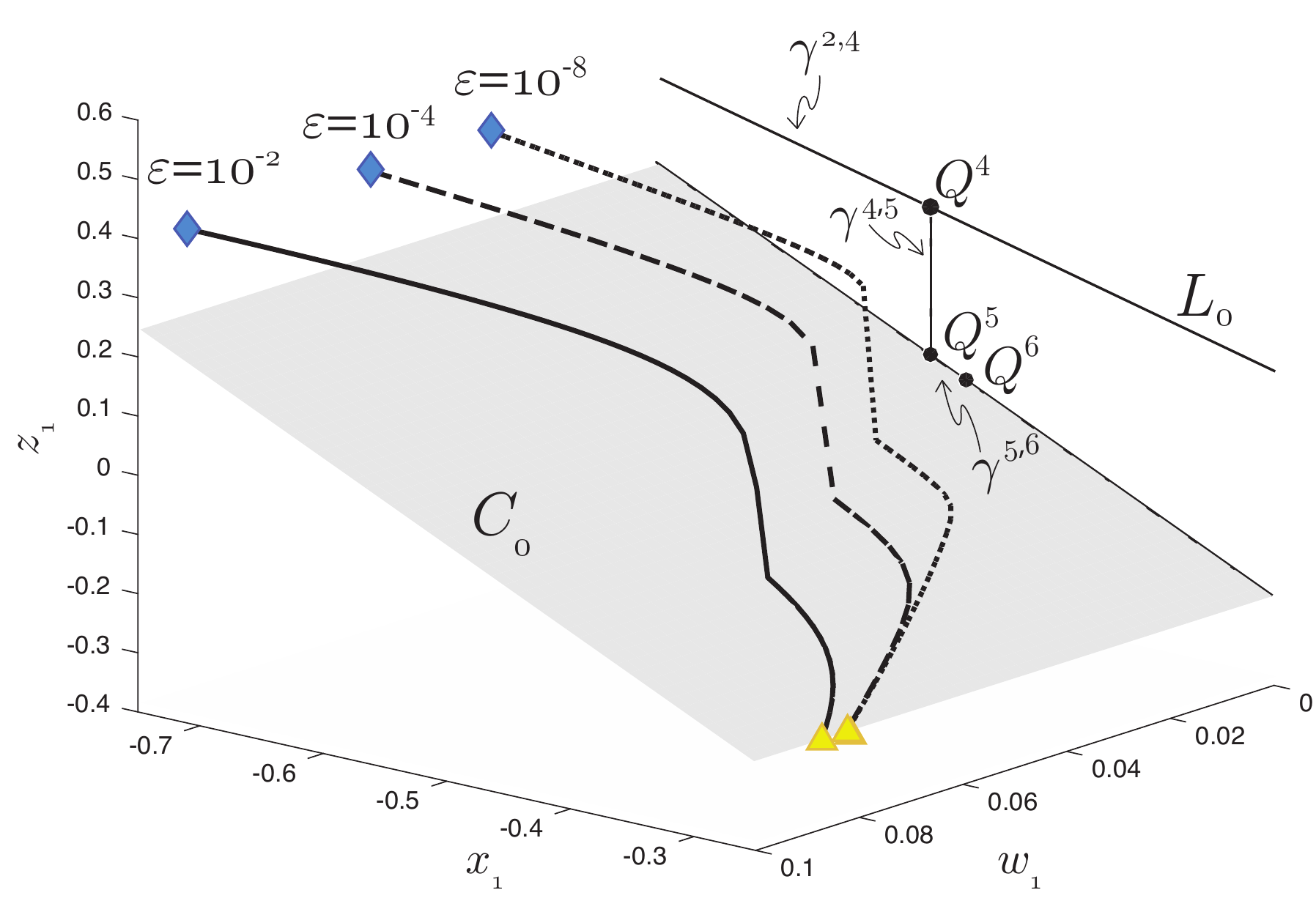}\end{subfigure}
\caption{Figure \subref{fig:num_closedorbit}: numerical simulation of \eqref{eq:3Dproblem} for $\vare\in\{10^{-8},10^{-4},10^{-2}\}$, $\alpha=0.9$ and $\xi=0.5$. In \subref{fig:num_chartK3}: portion of $\Gamma_\vare$ visible in chart $K_3$, i.e. between the green lower triangle and the red square. In \subref{fig:numericsK1_3D}: portion of $\Gamma_\vare$ visible in  $K_1$, i.e. between the blue diamond and the yellow upper triangle. We remark that the portion between the blue triangle and the red square is visible  both in  $K_3$ and $K_1$ since the two charts overlap for $y_3>0$ or $z_1>0$. Colours are available in the online version. }\label{fig:closedorbit}
\end{figure}

\section{Identification of the segments of $\Gamma_0$ at infinity}\label{sec:proof}
In this section we focus on the identification of the segments of $\Gamma_0$   \eqref{eq:gamma_1a}--\eqref{eq:gamma_c3}. We are especially interested in revealing the line $L_0$  and the segments that interact with it. In \ref{sec:dirchart_K3} we study the dynamics along chart $K_3$ and then in  \ref{sec:barkappa1} we consider chart $K_1$. More details are  available in \cite{Bossolini2015a}.
\subsection{Chart ${K}_3$}\label{sec:dirchart_K3}

We obtain the vector field in chart ${K}_3$  by inserting condition \eqref{eq:R&S_kappa3_3d} into the fast problem \eqref{eq:fast_problem_2}. This vector field is de-singularized at $w_3=0$   by division of $\rme^{1/w_3}$. For the sake of readability we drop the subscripts:

\be \label{eq:kap3_small}
\begin{aligned}
\dot{w} &=  w  \rme^{-\frac{2}{w}} \left( y + \frac{x+1}{\xi} \right),\\
\dot{x} &= -\vare (x+1+ \alpha ) + x \rme^{-\frac{2}{w}} \left( y + \frac{x+1}{\xi} \right),\\
\dot{y} &= \vare w(1-\rme^{-\frac{1}{w}})  + y \rme^{-\frac{2}{w}} \left( y + \frac{x+1}{\xi} \right),\\
\dot{\vare} &= 0.
\end{aligned}\ee
System  \eqref{eq:kap3_small} is a four-dimensional vector field defined on $\mathbb{R}^4$ where we treat  $\vare$ as a  variable.  The set $w=\vare=0$   consists of non-hyperbolic fixed points of \eqref{eq:kap3_small} and the two lines $C_{0,\infty}$ and $L_0$ \eqref{eq:L0C0infty} are contained within this set.  Since we consider a regime   of $w$ sufficiently small, we approximate $1 - \rme^{-1/w} \simeq 1 $ in the $y$-equation  of \eqref{eq:kap3_small} to simplify the computations. Qualitatively this has no effects on the results.\\
We blow-up   \eqref{eq:kap3_small} around  $Q^1$ in order to extend the hyperbolicity of   $C_0$ up to infinity. To do so we need to get rid of the  exponential terms. We deal with it  by introducing a new variable $q$:
\be \label{eq:qdef}
q = \rme^{-\frac{2}{w}},\ee
so that the extended system contains only algebraic terms in its variables \cite{Kristiansen2015a}.  Indeed by differentiating   \eqref{eq:qdef} with respect to time we obtain: 
\be\label{eq:qdot}
\begin{aligned}
\dot{q} &= 2w^{-2}\dot{w} \rme^{-\frac{2}{w}},\\
&= 2 w^{-1} q \left(y + \frac{x+1}{\xi}\right),
\end{aligned}\ee
where we have used \eqref{eq:kap3_small} and \eqref{eq:qdef}. Inserting \eqref{eq:qdot} into \eqref{eq:kap3_small} we obtain the five-dimensional vector field:
\be \label{eq:kap3_global}
\begin{aligned}
\dot{w} &=  w^2 q \left( y + \frac{x+1}{\xi} \right),\\
\dot{x} &= -\vare w (x+1+\alpha) + x q  w \left( y + \frac{x+1}{\xi} \right),\\
\dot{y} &= \vare w^2 + y q w\left( y + \frac{x+1}{\xi} \right),\\
\dot{q} &= 2 q^2 \left( y + \frac{x+1}{\xi} \right),\\
\dot{\vare} &= 0.
\end{aligned}\ee
after multiplying the right hand side by $w$.  
The  evolution of $q$ in \eqref{eq:kap3_global} is slaved by $w$ through \eqref{eq:qdef}. However, this dependence is not explicit and we will refer to it only when needed. 
We refer to \cite{Kristiansen2015a} for further details on this approach. 
\noindent System \eqref{eq:kap3_global} has a 3-dimensional space of non-hyperbolic fixed points for $\vare=q=0$, since each  point has a quintuple zero eigenvalue.  To overcome the degeneracy we introduce the blow-up map:
\be \label{eq:blowup_qeps}
\quad q = \bar{r}\bar{q}, \quad \vare = \bar{r}\bar{\epsilon},
\ee
with $(\bar{q},\bar{\epsilon})\in   S^1 $ and $\bar{r}\geq0$ while the variables $(w,x,y)\in \mathbb{R}^3$ in \eqref{eq:kap3_global} are kept unchanged. We remark that the quantity $\vare$ in \eqref{eq:blowup_qeps} is a constant, hence the blown-up space is foliated by invariant hyperbolas.
We  study the  two local charts:
\begin{subequations}
\label{eq:charts_compactification}
 \begin{flalign} 
&\mathcal{K}_1: \quad q = r_1 , \qquad \vare = r_1 \epsilon_1,\label{eq:kappa1bar} \\
&\mathcal{K}_2: \quad q = r_2 q_2 , \qquad \vare = r_2 .\label{eq:kappa2bar}
\end{flalign}
\end{subequations}
Notice that   $q_2=\Or(1)$ in chart $\mathcal{K}_2$ corresponds to $w = \Or (\ln^{-1} \vare^{-1})$ or  $z_2 = \Or(\ln \vare^{-1})$ through \eqref{eq:qdef}. This is  the relevant regime for the na\"ive identification of $L_0$ as in  \eqref{eq:dynamicsinfinity}.
\paragraph{Chart $\mathcal{K}_1$}
To simplify the  analysis we place the $x$-axis of  \eqref{eq:kap3_global} on  $C_0$ by introducing the new coordinate $ \tilde{x} = x  + \xi y  + 1 $
so that    $Q^1$  is now in the origin of chart ${K}_3$.  
We insert  \eqref{eq:kappa1bar} into  \eqref{eq:kap3_global}  and divide out a common  factor of $r_1$ to obtain the de-singularized system in chart $\mathcal{K}_1$. This system is independent of  $r_1$ therefore we restrict the analysis to the remaining four variables $(w,\tilde{x},y,\epsilon_1)$ and we drop the subscript.  
The origin of the reduced system is still degenerate with all zero eigenvalues.  To overcome the degeneracy  we introduce the following blow-up of $C_{0,\infty}$:
\[
w = \bar{r} \bar{w} , \quad \tilde{x} = \bar{r}\bar{\tilde{x}} , \quad \epsilon = \bar{r} \bar{\epsilon},
\]
where $(\bar{w},\bar{\tilde{x}},\bar{\epsilon})\in  {S}^2 $ and $\bar{r}\geq0$ small, while the variable $y\in \mathbb{R}$ is kept unchanged. We study   charts $\mathtt{K}_1$ and $\mathtt{K}_2$  that are defined for $\bar{w}=1$ and  $\bar{\epsilon}=1$ respectively.\\
Chart $\mathtt{K}_1$ has an attracting 3-dimensional center manifold $M_1$ in the origin.  This manifold   is the extension of the slow-manifold $S_\vare$ (see Proposition \ref{prop:hopf_perturbed}) into  chart $\mathtt{K}_1$ when $\vare =\text{const.}$ and $q=\rme^{-2/w}$.
Thus  we can  extend  the hyperbolicity of $C_0$  up to $C_{0,\infty}$  for $\bar{\epsilon}=0$ and   recover  the  contraction to $Q^1$ of Figure \ref{fig:alpha_larger_xi}.  
We follow the unique unstable direction of $Q^1$ that sits on the sphere $\bar{r}=0$. This direction exits chart  $\mathtt{K}_1$ for $\bar{\epsilon}$ large and  contracts to the origin of chart $\mathtt{K}_2$ along the invariant plane $\bar{r}=0$. Using hyperbolic methods we  follow the unique 1-dimensional unstable manifold $\gamma_2^1$ departing from the origin of chart $\mathtt{K}_2$  into chart   $\mathcal{K}_2$, where it enters with $q_2$ small.

\paragraph{Chart $\mathcal{K}_2$} 
We substitute  \eqref{eq:kappa2bar} into \eqref{eq:kap3_global} and divide  the right-hand side by   $r_2$  to obtain the dynamics in chart $\mathcal{K}_2$. The system is independent of  $r_2$ and we restrict the analysis to   $(w,x,y,q_2)$. The unstable manifold  $\gamma_2^1 :=   \left\{ (w,x,y,q_2)\in  \mathcal{K}_2 \lvert   \quad (w,x,y) = (0,-1,0), q_2 \geq 0  \right\}$
 contracts towards the fixed point:
 \be\label{eq:point_fixed}
 (w,x,y,q_2) = ( 0, -1, 0, 0  ).
 \ee
 The point    \eqref{eq:point_fixed}  belongs to a   plane of  non-hyperbolic fixed points with $w=q_2=0$ and to overcome the loss of hyperbolicity  we   introduce the   blow-up map (after having dropped the subscript):
\be \label{eq:blowup_k2_wq}
w = \bar{r} \bar{w} , \quad q = \bar{r}\bar{q} ,
\ee
where $(\bar{w},\bar{q})\in  {S}^1$ and $\bar{r} \geq0$. We study    charts  $\hat{\mathtt{K}}_1$ and  $\hat{\mathtt{K}}_2$ that are defined for $\bar{w}=1$ and $\bar{q}=1$ respectively.

\paragraph{Chart $\hat{\mathtt{K}}_1$}
We insert \eqref{eq:blowup_k2_wq} with $\bar{w}=1$ into the  vector field of chart $\mathcal{K}_2$ and drop the bar. We  divide the system by  $r$ to obtain the de-singularized equations:
\be\label{eq:eq_kappa1hat}
\begin{aligned}
\dot{r} &= r^2 q \left( y+ \frac{x+1}{\xi}\right),\\
\dot{x} &= -(x+1+\alpha)+ xrq \left( y+ \frac{x+1}{\xi}\right),\\
\dot{y} &= r + y rq \left( y+ \frac{x+1}{\xi}\right),\\
\dot{q} &= q^2  (2-r)\left( y+ \frac{x+1}{\xi}\right).
\end{aligned}\ee
In the following important lemma we identify the line $L_0$ and the segment $\gamma^{1,2}$:  
\begin{lemma}\label{lem:lineL0}
In chart $\hat{\mathtt{K}}_1$  there exists an attracting 3-dimensional center manifold:
\be\label{eq:cm_hatkappa1}
x = -1-\alpha + \Or(r+q),
\ee 
whose intersection with the plane $r=q=0$ corresponds to the line $L_0$ \eqref{eq:L0_K3}.
The trajectory $\gamma^{1,2}$ defined in \eqref{eq:gamma_1a} connects  along a stable fiber the point \eqref{eq:point_fixed} to  $Q^2$  \eqref{eq:Qa}. 
\end{lemma}

\begin{proof}   \eqref{eq:eq_kappa1hat} has a line of fixed points for $r=q=0, x=-1-\alpha, y\in\mathbb{R}$. This line corresponds to $L_0$ through the coordinate changes \eqref{eq:kappa2bar}, \eqref{eq:blowup_k2_wq}. The linearized dynamics on $L_0$ is hyperbolic only in the  $x$-direction and furthermore is stable. Therefore  \eqref{eq:cm_hatkappa1} appears for $r,q$ sufficiently small. The point \eqref{eq:point_fixed} in chart $\hat{\mathtt{K}}_1$ becomes:
\be\label{eq:point_fixed_K1}
 (r,x,y,q) = ( 0, -1, 0, 0  ),
 \ee
 hence there is a solution backwards asymptotic to \eqref{eq:point_fixed_K1} and forward asymptotic to  $Q^2\in L_0$ (recall \eqref{eq:Qa}) through a stable fiber. This connection corresponds to $\gamma^{1,2}$.\qed\end{proof}
We insert \eqref{eq:cm_hatkappa1} into \eqref{eq:eq_kappa1hat} to obtain the dynamics on the center manifold. The resulting vector field has a  line of  non-hyperbolic fixed points, corresponding to $L_0$,  for $r=q=0$  since each point has a triple zero eigenvalue. We gain hyperbolicity of this line by introducing the blow-up map:
\be \label{eq:blowup_k2_rq}
r = \rho \sigma , \quad q = \rho ,
\ee 
where $\rho\geq0, \sigma\geq0$. In chart \eqref{eq:blowup_k2_rq} the point  $Q^2$ \eqref{eq:Qa} is blown-up to the $\sigma$-axis $\{y=\rho=0, \sigma\geq 0\}$. Similarly $Q^4$ \eqref{eq:Qb} corresponds to the line $\{y=2\alpha/\xi, \rho=0, \sigma\geq 0\}$. We divide the vector field of chart \eqref{eq:blowup_k2_rq} by the common divisor $\rho$ and obtain:
\be\label{eq:cm_kappa1check}
\begin{aligned}
\dot{\sigma} &= \sigma(-2+\rho \sigma + \rho^2)  \left( y - \frac{\alpha}{\xi} \right)\left( 1 + \Or\left(\rho\right) \right),\\
\dot{y} &= \sigma + y \rho \sigma \left( y - \frac{\alpha}{\xi} \right)\left( 1 + \Or\left(\rho\right) \right),\\
\dot{\rho} &= \rho(2-\rho \sigma) \left( y - \frac{\alpha}{\xi}\right)\left( 1 + \Or\left(\rho\right) \right).
\end{aligned}\ee
Following equations \eqref{eq:point_fixed} and $\gamma_2^1$ we enter  chart \eqref{eq:blowup_k2_rq} with $\sigma=0$ and $y=0$. Subsequently, by following $\gamma^{1,2}$ we have $\rho=0$. In the following we describe the dynamics within $L_0$ and identify $\gamma^{2,4}$ as a heteroclinic orbit. 
\begin{lemma}\label{lem:lineL0_part2}
System \eqref{eq:cm_kappa1check} has two invariant planes for $\rho=0$ and $\sigma=0$. Their intersection $\rho=\sigma=0$ is a line of fixed points. We have:
\begin{itemize}
\item The origin  $(\sigma,y,\rho)=(0,0,0)$ has a strong stable manifold:
\be \label{eq:gamma_1^1}
W^s(0,0,0) := \left\{ (\sigma,y,\rho)\in \mathbb{R}^2 \times\mathbb{R}_{+}  \lvert \quad \sigma=0,   y=0, \rho\geq 0  \right\}.
\ee
\item There exists a heteroclinic connection: 
\be \label{eq:explicit_gamma_ab} \gamma^{2,4} = \left\{ (\sigma,y,\rho)\in \mathbb{R}^2 \times\mathbb{R}_{+}  \lvert \quad \sigma=2\frac{\alpha}{\xi} y - y^2,   y\in(0,2\alpha/\xi), \rho= 0  \right\}. \ee 
 joining $(\sigma,y,\rho)=(0,0,0)$ backwards in time with $(\sigma,y,\rho) = (0,2\alpha/\xi,0)$ forward in time. 
\item The point $(\sigma,y,\rho) = (0,2\alpha/\xi,0)$ has a strong unstable manifold:
\be \label{eq:gamma_1^2}
W^u (0,2\alpha/\xi,0) := \left\{ (\sigma,y,\rho)\in \mathbb{R}^2 \times\mathbb{R}_{+}  \lvert \quad \sigma=0,   y=2\alpha/\xi, \rho\geq 0  \right\}.
\ee
\end{itemize}
%
  \end{lemma}
The results of Lemma \ref{lem:lineL0_part2} are summarized in  Figure \ref{fig:center_manifold}.
\begin{remark}
Upon blowing down, the expression in \eqref{eq:explicit_gamma_ab} gives $\gamma^{2,4}$ in  \eqref{eq:gamma_ab}. We use the same symbol in \eqref{eq:explicit_gamma_ab} and \eqref{eq:gamma_ab} for simplicity.
\end{remark} 
\begin{figure}[h!]
\centering
\includegraphics[scale=0.8]{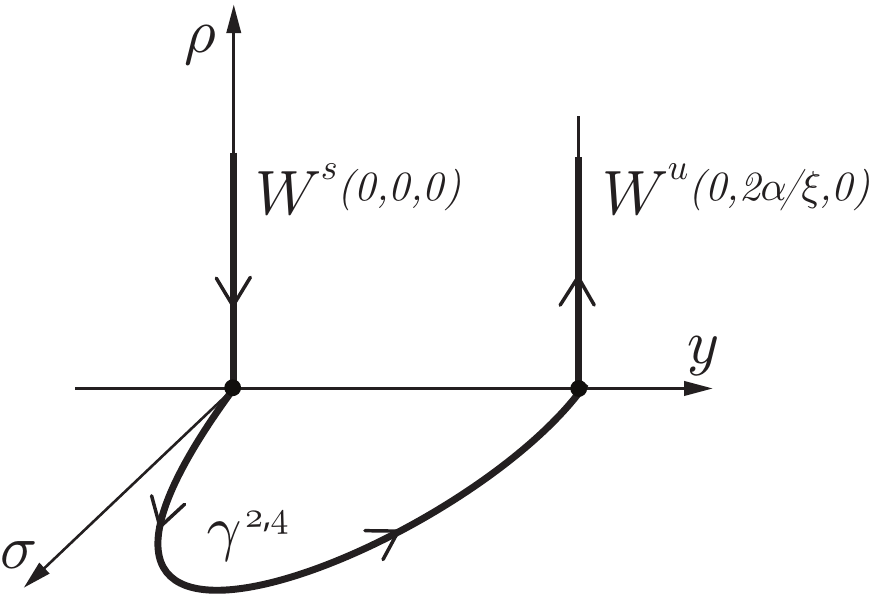}
\caption{Dynamics in chart $\hat{\mathtt{K}}_1$. The plane $\rho=0$ corresponds to the blown-up line $L_0$. Here the singular dynamics contracts to $L_0$ and then is  expelled from it.}\label{fig:center_manifold}
\end{figure} 
\begin{proof}
On the invariant plane $\sigma=0$ we have the following dynamics:
\be\label{eq:cm_kappa1checkr}
\begin{aligned}
\dot{y} &= 0,\\
\dot{\rho} &= 2\rho\left( y - \frac{\alpha}{\xi}\right)\left( 1 + \Or\left(\rho\right) \right).
\end{aligned}\ee
This plane is foliated with invariant lines in the $y$-direction.  The solution  of \eqref{eq:cm_kappa1checkr} with $y=0$ is  \eqref{eq:gamma_1^1} and contracts towards the invariant plane $\rho=0$. Hence this trajectory acts as a strong stable manifold.  We substitute  $\rho=0$ into \eqref{eq:cm_kappa1check} and after dividing by $\sigma$ we obtain the  explicit solution  \eqref{eq:explicit_gamma_ab} given the initial condition in the origin.  This solution is forward asymptotic to 
 $(\sigma,y,\rho) = (0,2\alpha/\xi,0)$.  Eventually  $\rho$ expands on the strong unstable manifold \eqref{eq:gamma_1^2}, that is the solution of   \eqref{eq:cm_kappa1checkr} with $y=2\alpha/\xi$. 
\qed\end{proof}
Using hyperbolic methods we  follow the unstable manifold $W^u(0,2\alpha/\xi,0)$ into chart  $\hat{\mathtt{K}}_2$ where it contracts towards the origin along the invariant plane $\bar{r}=0$. We continue this trajectory by following the unstable manifold of the origin on the  plane $\bar{w}=0$. We continue this manifold into chart $\mathcal{K}_1$, since eventually   chart $\mathcal{K}_2$ is no longer suited  to describe this trajectory.
Here the variable $\epsilon$ decreases exponentially and for $\epsilon=w=0$ we obtain a layer problem. 
 Therefore from $Q^4$ we follow after de-singularization a fast fiber $\gamma^{4,5}$ that corresponds to the solution of  this layer problem   and that contracts to the point $Q^5$ on $C_0$.
Since  $Q^5$ may not be visible in chart ${K}_3$, we compute its coordinates in chart  ${K}_1$.

\subsection{Chart ${K}_1$}\label{sec:barkappa1} 
We insert \eqref{eq:R&S_kappa21_3d} into the fast problem \eqref{eq:fast_problem_2} and divide by $\rme^{z_1/w_1}$ to obtain the de-singularized vector field in chart ${K}_1$. We drop the subscript henceforth for the sake of readability:
\be \label{eq:barkap1_global}
\begin{aligned}
\dot{w} &=  - \vare w^2 (1 - \rme^{-z/w}),\\
\dot{x} &= -\vare \left(x+(1+ \alpha )z\right) - \vare x w (1 - \rme^{-z/w}),\\
\dot{z} &=  - \rme^{-2z/w} \left( 1 + \frac{x+z}{\xi} \right) - \vare z w (1 - \rme^{-z/w}).
\end{aligned}\ee
In chart ${K}_1$ the layer problem is obtained by requiring $\vare=0$ in \eqref{eq:barkap1_global}. Hence the dynamics on the layer problem is only in the $z$-direction and  the fibers are all vertical. In particular the fiber $\gamma^{4,5}$  is written as in \eqref{eq:gamma_bc} since it departs from $K_{31}(Q^4)$.
It follows that   $\gamma^{4,5}$ is forward asymptotic to the point $Q^5$ defined in \eqref{eq:Qc}. The point $Q^5$ is connected to $Q^6$ through the segment $\gamma^{5,6}$, according to the analysis of the reduced problem of section \ref{sec:compact_reduced}.  From the point $Q^6$ the solution is connected to the point $Q^1$ through the manifold $W^{c,u}$. This closes the singular cycle $\Gamma_0$. \\
Figure  \ref{fig:kappa1_new}  illustrates the dynamics in chart ${K}_1$. We remark that the change of coordinates from chart ${K}_3$ to chart ${K}_1$ is defined for $z_1>0$ and therefore when $\alpha>1$ the point $Q^5$ is visible only in chart ${K}_1$. 


\section{Outline of a proof}\label{sec:global_results}
To prove  Conjecture \ref{con:conjecture_thm} we would have to consider a section $\Lambda_1:= \{w_1=\delta\}$ transverse to $W^{c,u} $ where $\delta>0$ is small but fixed. Using the blow-up in chart ${K}_3$ we can track a full neighbourhood $N\subset \Lambda_1$ of $ \Lambda_1 \cap W^{c,u} $ using  Proposition \ref{theo:bifurcation_alpha}, $\gamma^{1,2}, \gamma^{2,4}, \gamma^{4,5}, \gamma^{5,6}$ and $W^{c,u} $ respectively, to obtain a return map $P_1: N \to N$ for $\vare$ sufficiently small. For $\vare = 0$ the forward flow of $N$ contracts to the point $Q^1$. This would provide the desired contraction of $P_1$ and establish, by the contraction mapping theorem, the existence of the limit cycle $\Gamma_\vare$ satisfying $\Gamma_\vare\to\Gamma_0$ for $\vare\to0$.  

\section{Conclusions}\label{sec:conclusion}
We have  considered the one dimensional spring-block model that describes the earthquake faulting phenomenon. We have  used   geometric singular perturbation theory and the blow-up method to provide a detailed description of the periodicity of the earthquake episodes, in particular we have untangled the increase in  amplitude  of the cycles for $\vare\to 0$ and their   relaxation oscillation structure.
We have shown that the limit cycles arise from a degenerate Hopf bifurcation. The degeneracy is due to an underlying Hamiltonian structure that leads to large amplitude oscillations. Using the Poincar\'e compactification together with the blow-up method, we have described how these limit cycles behave near infinity in the limit of $\vare\to0$.  A  full detailed  proof of Conjecture \ref{con:conjecture_thm}, including the required careful estimation of the contraction, will be the subject of a separate manuscript. \\
We have observed that the notation of quasi-static slip motion to define the reduced problem \eqref{eq:reduced_problem} is misleading. Indeed the solutions of \eqref{eq:reduced_problem}  have an intrinsic slow-fast structure  resembling the  stick-slip oscillations. Our analysis also shows that  the periodic solutions of \eqref{eq:3Dproblem} cannot be investigated by studying the so-called quasi-static slip phase and the  stick-slip phase separately, as it is done in  \cite{Ruina1983,gu1984a}, since the two phases are connected by the  non-linear terms of  \eqref{eq:3Dproblem}. We also suggest suitable coordinate sets and time rescales to deal with the stiffness of   \eqref{eq:3Dproblem} during numerical simulations.  We hope that a deeper understanding of the structure of the earthquake cycles may be of help to the temporal predictability of the earthquake episodes.\\
We presuppose that we can apply  some of the ideas in this manuscript  to the study of the 1-dimensional spring-block model with Dieterich state law. Indeed in this new system   the fixed point in the origin behaves like a saddle,  the critical manifold  loses  hyperbolicity  like \eqref{eq:normally_attracting} and solutions reach infinity in finite time for $\vare \to0$. Moreover we think that these ideas  can also be used to study the  continuum formulation of the Burridge and Knopoff model with Ruina state law, in particular to analyse    the self-healing slip pulse solutions \cite{heaton1990a}.  Indeed this latter model has the same difficulties of  \eqref{eq:3Dproblem} in terms of  small parameter and   non-linearities of the vector field \cite{Erickson2011}. We remark that the self-healing slip pulse solutions  are considered to be related to the energy of an earthquake rupture. \\

 
\ack
The first author  thanks Thibault Putelat and Bj\"orn Birnir for the useful discussions. We acknowledge the Idella Foundation for supporting the research. This research was partially done whilst the first author was a visiting Researcher at the Centre de Recerca Matem\`atica in the Intensive Research Program on Advances in Nonsmooth Dynamics.


\section*{References}
\bibliographystyle{jphysicsB}
\bibliography{bibliography/Earthquake_Article}                         
\end{document}